\documentclass[11pt]{article}
\usepackage{amsfonts}
\usepackage{mathrsfs}
\usepackage{enumitem}
\usepackage{bbm}
\usepackage{amsfonts}
\usepackage{amssymb,amsmath,amsthm,graphicx}
\usepackage{float} \usepackage[colorlinks=true]{hyperref}
\hypersetup{urlcolor=blue, citecolor=red}

\newcommand{\R}{{\mathbb R}}

\newcommand{\Rd}{\R^d}

\newtheorem{assumption}{Assumption}[section]

\allowdisplaybreaks

\begin{document}
\title{{The rigorous derivation of Vlasov equations with local alignments from moderately interacting particle systems}}
\author{Jinhuan Wang, Mengdi Zhuang, Hui Huang
{\thanks{Corresponding author. E-mail: hui.huang@uni-graz.at}}\\
{\small School of Mathematics and Statistics, Liaoning University, Shenyang 110036, PR China}\\
{\small Department of Mathematics and Scientific Computing, University of Graz, Heinrichstra\ss e 36, 8010 Graz, Austria}
} \maketitle
\date{}
\newtheorem{theorem}{Theorem}
\newtheorem{definition}{Definition}[section]
\newtheorem{lemma}{Lemma}[section]
\newtheorem{proposition}{Proposition}[section]
\newtheorem{corollary}{Corollary}
\newtheorem{remark}{Remark}
\renewcommand{\theequation}{\thesection.\arabic{equation}}
\catcode`@=11 \@addtoreset{equation}{section} \catcode`@=12
\maketitle{}
\begin{abstract}
In this paper, we present a rigorous derivation of the mean-field limit for a moderately interacting particle system in $\R^d$ $(d\geq 2)$. For stochastic initial data, we demonstrate that the solution to the interacting particle model, with an appropriately applied cut-off, converges in probabilistic sense to the solution of the characteristics of the regularized Vlasov models featuring local alignments and Newtonian potential. Notably, the cutoff parameter for the singular potential is selected to scale polynomially with the number of particles, representing an improvement over the logarithmic cut-off obtained in \cite{HWL2024}.
\begin{description}
\item[Mathematics Subject Classification:] 35Q83  82C22 35Q70  60K35
\item[Keywords:] Mean-field limit; Propagation of chaos; Local alignment; Newtonian potential; Cucker-Smale model
\end{description}
\end{abstract}


\section{Introduction}
The purpose of this paper is to study the mean-field  limit of a stochastic, moderately interacting particle system in order to derive the following Vlasov type equation with a local alignment term
\begin{align}\label{0.3}
	\partial_tf+v\cdot\nabla_xf-\nabla_v\cdot((\gamma v+\lambda(\nabla_xV+\nabla_xW\star\rho))f)=\nabla_v\cdot(\beta(v-u)f)
\end{align}
subject to the initial data $f(t,x,v)|_{t=0}=:f_0(x,v)$, $(x,v)\in \mathbb{R}^d\times \mathbb{R}^d$. Here $f(t,x,v)$ represents the particle density at $(x,v)\in \mathbb{R}^d\times \mathbb{R}^d$ and at time $t\in  \mathbb{R}_+$, $u$ is the local particle velocity, i.e.
$$
u(t,x):=\frac{\int_{\mathbb{R}^d}vf(t,x,v)\,dv}{\rho(t,x)} \mbox{ with }\rho(t,x):=\int_{\mathbb{R}^d}f(t,x,v)dv\,,$$
$V(x)$ and $W(x)$ are the confinement and the interaction potentials respectively, which shall be specified later. The first two terms on the left side of the equation \eqref{0.3} represent the free transport of the particles. The third term accounts for linear damping with a positive strength parameter $\gamma>0$. In equation \eqref{0.3}, we choose $V(x)=\frac{|x|^2}{2}$ and $\nabla_xW(x)=C_d\frac{x}{|x|^d}$. Additionally, the particle confinement and interaction forces in position are incorporated through potentials $V(x)$ and $W(x)$ with a positive strength parameter $\lambda>0$.  The right-hand side of \eqref{0.3} corresponds to the local alignment force for particles, which was introduced in \cite{KMT2013} as part of swarming models, and $\beta>0$ is a constant.

Equation \eqref{0.3} can be traced back to the model introduced by Cucker and Smale \cite{CS12007,CS22007}, which serves as a simple dynamical system to explain the emergence of flocking mechanisms in systems such as birds, fish, etc. The system considers an interacting $N$-particle system taking the form of
\begin{equation}\label{0.0}
\frac{d}{dt}x_i=v_i\quad\frac{d}{dt}v_i=-\frac{1}{N}\sum_{j=1}^NK_0(x_i,x_j)(v_j-v_i),\quad i=1,\dots,N,
\end{equation}
where each particle is described by its position $x_i\in\mathbb{R}^d$ and velocity $v_i\in\mathbb{R}^d$ with $d\geq 1$, and $K_0$ is a smooth  symmetric kernel. If we consider the empirical distribution function
\begin{align}
f(t,x,v)=\frac{1}{N}\sum_i^N\delta(x-x_i)\delta(v-v_i)\,,
\end{align}
then it will satisfy the following kinetic equation
\begin{align}\label{0.1}
\partial_tf+v\cdot\nabla_xf+\nabla_v\cdot(fL[f])=0, \quad in\quad \mathbb{R}^d\times\mathbb{R}^d\times(0,T),
\end{align}
where the operator $L$ has the form
\begin{align}\label{0.1'}
L[f]:=\int_{\R^{2d}}K_0(x,y)f(y,w)(w-v)\,dwdy.
\end{align}

The Cucker-Smale (CS) model has garnered extensive attention in various fields (see for example \cite{CCDW2015,CFRT2010,CFTF2010,CR2010,CCR2011,CS12007,CS22007,DFT2010,FHT2011,HJNXZ2017,HLSX2014,
	HL2009,HT2008,MT2011}). Especially, the existence of classical solutions to the CS equation has been established in \cite{HT2008}, and the mean-field limit of the CS particle model has been derived in \cite{CZ2021,HL2009}.
In \cite{KMT2013}, in order to prevent mass loss at infinity, a confinement potential $V$ is introduced into the equation \eqref{0.1}. This potential is designed such that $V\to\infty$ as $|x|\to\infty$. As a result, the following equation is obtained
\begin{align*}
\partial_tf+v\cdot\nabla_xf-\nabla_v\cdot(f\nabla_xV)+ \nabla_v\cdot(fL[f])=0, \quad in\quad \mathbb{R}^d\times\mathbb{R}^d\times(0,T).
\end{align*}
In \cite{MT2011}, Motsch and Tadmor noticed that the normalization factor $\frac{1}{N}$ in \eqref{0.0} can lead to some undesirable characteristics. In particular, when a small group of particles is situated far away from a much larger group of particles, the internal dynamics within the small group become almost halted due to the total large number of particles.
 To address this issue, they proposed a modified Cucker-Smale model that replaces the original CS alignment with a normalized non-symmetric alignment
 \begin{align}\label{0.1"}
 \tilde L[f]:=\frac{\int_{\R^{2d}}\phi(x-y)f(y,w)(w-v)dwdy}{\int_{\R^{2d}\times\R^d}\phi(x-y)f(y,w)dwdy}\,.
 \end{align}
  Subsequently, in \cite{KMT2013}, Karper, Mellet, and Trivisa introduced a new model by  additionally considering the singular limit where the Motsch-Tadmor  flocking kernel $\phi$ in $\tilde L[f]$ converges a Dirac distribution. The Motsch-Tadmor correction then converges to a local alignment term
$$
\frac{j-\rho v}{\rho}=u-v,
$$
where
$
j(x,t):=\int_{\mathbb{R}^d}vf(t,x,v)\,dv\,.
$
By incorporating this additional term, the following CS model with strong local alignment, as introduced in \cite{KMT2013}, is obtained
\begin{align}\label{0.2}
\partial_tf+v\cdot\nabla_xf-\nabla_v\cdot(f\nabla_xV)+ \nabla_v\cdot (fL[f])+\beta \nabla_v\cdot(f(u-v))=0\,,
\end{align}
where $\beta>0$ is a constant.
When $V$ satisfies that $V\to\infty$ as $|x|\to\infty$, and $L$ or $\tilde L$ satisfies \eqref{0.1'} or \eqref{0.1"}, the global existence of weak solutions to \eqref{0.2} has been proved in \cite{KMT2013}. If $V(x)=\lambda\frac{|x|^2}{2}\,,L[f]=-\gamma v-\lambda(\nabla_xW\star\rho)$ and $\nabla_xW(x)=C_d\frac{x}{|x|^d}$, the global existence of weak solutions to \eqref{0.3} also has been proved in \cite{CC2020} and \cite{HWL2024}. 

 As in \cite{HWL2024}, we aim to investigate the mean-field limit of the following moderately interacting particle system which will lead to the Vlasov equation \eqref{0.3}
 \begin{equation}\label{particle}
 	\left\{
 	\begin{aligned}
 		&\frac{dx_i^{\varepsilon,\delta}}{dt}=v_i^{\varepsilon,\delta},\quad i=1,\dots,N,\\
 		&\frac{dv_i^{\varepsilon,\delta}}{dt}=-\frac{\lambda}{N}\sum_{j=1}^N\nabla_x W^\varepsilon(x_i^{\varepsilon,\delta}-x_j^{\varepsilon,\delta})
 		-(\gamma v_i^{\varepsilon,\delta}+\lambda\nabla_x V(x_i^{\varepsilon,\delta})+\beta v_i^{\varepsilon,\delta})
 		+\beta u^{\varepsilon,\delta}(x_i^{\varepsilon,\delta})\,,
 	\end{aligned}
 	\right.
 \end{equation}
where
\begin{equation*}
	\nabla_x W^\varepsilon(x)=\left\{
	\begin{aligned}
		&C_d\frac{x}{|x|^d},&&\quad|x|\ge\varepsilon,\\
		&C_d\varepsilon^{-d}x,&&\quad|x|<\varepsilon
	\end{aligned}
	\right.
\end{equation*}
is the regularization of the singular interaction force $\nabla_x W$
and
\begin{equation}\label{local}
{u}^{\varepsilon,\delta}({x}_i^{\varepsilon,\delta})=
\frac{\frac{1}{N}\sum\limits_{j=1}^N\phi_R({v}_i^{\varepsilon,\delta})\phi^\varepsilon({x}_i^{\varepsilon,\delta}-{x}_j^{\varepsilon,\delta})}{\frac{1}{N}\sum\limits_{j=1}^N\phi^{\varepsilon}({x}_i^{\varepsilon,\delta}-{x}_j^{\varepsilon,\delta})+\delta}
\end{equation}
is an approximation of the local velocity $u$. Here $\phi^{\varepsilon}:=\frac{1}{\varepsilon^d}\phi(\frac{x}{\varepsilon})$ with  $\phi$ being the standard mollifier. Then the function
$\phi^{\varepsilon}$ is $C_c^\infty(\mathbb{R}^d)$ and satisfies $\int_{\mathbb{R}^d}\phi^{\varepsilon} dx=1$. For the convenience of subsequent calculations, we define $\phi_R(v_j^{\varepsilon,\delta}):=v_j^{\varepsilon,\delta}\cdot h(\frac{|v_j^{\varepsilon,\delta}|}{R})$, where $h(r)\in C^2([0,\infty))$ satisfying
\begin{equation*}
	h(r)=\left\{
	\begin{aligned}
		&1,&&\quad r<1,\\
		&0,&&\quad r\ge2,
	\end{aligned}
	\right.
	\quad \mbox{ and }\quad  0\le h\le 1\,.
	\end {equation*}
Note that in the notation $({x}_i^{\varepsilon,\delta},{v}_i^{\varepsilon,\delta})$, we omit the superscript $R$ for the particles. This omission is because we will later choose $R=1/\delta$.

In this article we will show that equation \eqref{0.3} can be derived from particle system \eqref{particle} in the limit of $N\to \infty$, $\varepsilon,\delta \to 0$, with the scaling relation
\begin{equation}
	\varepsilon=N^{-\theta} \mbox{ and }\delta=R^{-1}=\frac{1}{\sqrt{\vartheta \ln(N)}}
\end{equation}
for some constants $\theta,\vartheta>0$.
To achieve this, we start by fixing $\varepsilon,\delta>0$. We then undertake a classical mean-field limit from \eqref{particle} to the auxiliary intermediate system given as follows:
{\small \begin{equation}\label{0.6}
\left\{
\begin{aligned}
&\frac{d\overline{x}_i^{\varepsilon,\delta}}{dt}=\overline{v}_i^{\varepsilon,\delta},\quad i=1,\dots,N,\\
&\frac{d\overline{v}_i^{\varepsilon,\delta}}{dt}=-\lambda\int_{\mathbb{R}^d}\nabla_xW^\varepsilon(\overline{x}_i^{\varepsilon,\delta}-y)\rho^{\varepsilon,\delta}(t,y)dy
-(r\overline{v}_i^{\varepsilon,\delta}+\lambda\nabla_x V(\overline{x}_i^{\varepsilon,\delta})+\beta \overline{v}_i^{\varepsilon,\delta})
+\beta\overline{u}^{\varepsilon,\delta}(\overline{x}_i^{\varepsilon,\delta})\,,\\
& f^{\varepsilon,\delta}(x,v,t)=\mbox{Law}(\overline{x}_i^{\varepsilon,\delta},\overline{x}_i^{\varepsilon,\delta}),\quad \rho^{\varepsilon,\delta}=\int_{\Rd} f^{\varepsilon,\delta}dv
\end{aligned}
\right.
\end{equation}}
with the same initial data as \eqref{particle} and
$$
\overline{u}^{\varepsilon,\delta}(\overline{x}_i^{\varepsilon,\delta})=
\frac{\int_{\R^{2d}}\phi_R(v)\phi^\varepsilon(\overline{x}_i^{\varepsilon,\delta}-y) f^{\varepsilon,\delta}(t,y,v)dvdy}{\int_{\R^{2d}}\phi^{\varepsilon}(\overline{x}_i^{\varepsilon,\delta}-y) f^{\varepsilon,\delta}(t,y,v)dvdy+\delta}\,.
$$
Note here that we consider $N$ independent copies $(\overline{x}_i^{\varepsilon,\delta},\overline{v}_i^{\varepsilon,\delta})$, $i=1,\dots,N$, and the intermediate system depends on $i$ only through the initial datum. Secondly, it is easy to check that the density function $f^{\varepsilon,\delta}$ shall satisfy the PDE
{\small \begin{align}\label{0.5} \partial_tf^{\varepsilon,\delta}+v\cdot\nabla_xf^{\varepsilon,\delta}-\nabla_v\cdot((\gamma v+\lambda (\nabla_xV+\nabla_xW^\varepsilon\ast\rho^{\varepsilon,\delta}))f^{\varepsilon,\delta})=\nabla_v\cdot(\beta(v-u^{\varepsilon,\delta})f^{\varepsilon,\delta})\,,
\end{align}}
which is a regularized version of our original model \eqref{0.3}. Finally, if we let $\varepsilon,\delta\to0$, one expects that $f^{\varepsilon,\delta}$ converges to $f$, which is exactly the solution to \eqref{0.3} in weak sense (See the literatures \cite{CC2020} and \cite{HWL2024} for detailed proof).

Mean-field limits from stochastic differential equations have been extensively studied since the 1980s. Comprehensive reviews by Golse \cite{G2003} and Jabin and Wang \cite{JW2017}, alongside the seminal works of Sznitman \cite{S1984, S1991}, provide a thorough overview of this field.  It has been established that in the limit of many particles, weakly interacting stochastic particle systems converge to a deterministic nonlinear process, as demonstrated by Oelschläger \cite{O1984}. He further extended this approach to systems of reaction-diffusion equations \cite{O1989} and porous-medium-type equations with quadratic diffusion \cite{oelschlager1990large} through the use of moderate interactions in stochastic partial-numerical simulations.
The particle system \eqref{particle} is categorized as moderately interacting due to the incorporation of a nonlocal term \eqref{local} to approximate local velocity dynamics.  Moderately interacting particle systems have also been instrumental in deriving a range of equations, including the porous medium equations with exponent 2 \cite{JM1998, philipowski2007interacting}, the diffusion-aggregation equation with a delta potential \cite{chen2018modeling}, cross-diffusion equations \cite{CDHJ2021, chen2019rigorous}, and the porous-medium equation with fractional diffusion \cite{chen2022analysis}.

This paper provides the detailed derivation form the moderately interacting particle system to its mean field limit or Vlasov equation. Without relying on the BBGKY hierarchy (\cite{EGKT2014,S2012}), we will rigorously derive the kinetic equation using a probabilistic method. This method has been previously demonstrated in various models: it was applied to a pedestrian flow model in \cite{CGY2017}, to diffusion-aggregation equations with bounded kernels in \cite{chen2023well}, and to the Vlasov-Poisson equation in \cite{huang2020mean,lazarovici2017mean}. The model proposed in this paper presents two significant challenges. The first issue arises from the singular nature of the interaction force. Another difficulty involves the treatment of the local alignment term $u$. The denominator of $u$ is the density function $\rho$, which may approach zero (a vacuum state). This scenario requires additional work on the estimates and, to our knowledge, has rarely been addressed before.

We now briefly outline our approach to achieving convergence between the particle system and the mean field equation. Initially, we utilize the moderately interacting particle system with a cutoff as detailed in \eqref{particle} as our foundational model. In Theorem \ref{th3.1}, we demonstrate that this system and the auxiliary intermediate system (Vlasov flow with a cutoff) \eqref{0.6} converge closely as $N \to \infty$. Subsequently, it can be verified that the density function $f^{\varepsilon,\delta}$ fulfills the requirements of the PDE \eqref{0.5}. Ultimately, by allowing $\varepsilon, \delta \to 0$, we anticipate that $f^{\varepsilon,\delta}$ will converge to $f$, which solves \eqref{0.3} (refer to the detailed proofs in \cite{CC2020} and \cite{HWL2024}). The steps of this procedure can be summarized as follows:
\\
\begin{center}
\setlength{\unitlength}{0.3cm}
\begin{picture}(11,11)

\linethickness{0.3mm}

\put(1,6){\vector(0,-1){3}}
\put(1,-10){\vector(0,-1){3}}
\put(1,-4){\vector(0,-1){3}}
\thinlines
\put(-15,10){\small${\frac{dx_i^{\varepsilon,\delta}}{dt}=v_i^{\varepsilon,\delta}
 },\quad i=1\dots,N$}
 \put(-15,7){\small${\frac{dv_i^{\varepsilon,\delta}}{dt}=-\frac{\lambda}{N}\sum\limits_{j=1}^N\nabla_x W^\varepsilon(x_i^{\varepsilon,\delta}-x_j^{\varepsilon,\delta})
 		-(\gamma v_i^{\varepsilon,\delta}+\lambda\nabla_x V(x_i^{\varepsilon,\delta})+\beta v_i^{\varepsilon,\delta})
 		+\beta u^{\varepsilon,\delta}(x_i^{\varepsilon,\delta})\qquad\qquad(a)}$}
 \put(1.5,4.5){\small${N\to\infty}$}
 \put(-15,-9){\small${\partial_tf^{\varepsilon,\delta}+v\cdot\nabla_xf^{\varepsilon,\delta}-\nabla_v\cdot((\gamma v+\lambda(\nabla_xV+\nabla_xW^\varepsilon\ast\rho^{\varepsilon,\delta}))f^{\varepsilon,\delta})=\nabla_v\cdot(\beta(v-u^{\varepsilon,\delta})f^{\varepsilon,\delta})~\qquad(c)}$}
 \put(1.5,-12){\small $\varepsilon,\delta\to 0$}
 \put(-15,-15){\small ${\partial_tf+v\cdot\nabla_xf-\nabla_v\cdot((\gamma v+\lambda(\nabla_xV+\nabla_xW\star\rho))f)=\nabla_v\cdot(\beta(v-u)f)\quad\qquad\qquad\qquad\qquad(d)}$}
 \put(1.5,-6){\small It\^ {o}'s formula}
 \put(-15,0){\small${
\frac{d\overline{x}_i^{\varepsilon,\delta}}{dt}=\overline{v}_i^{\varepsilon,\delta}},\quad i=1\dots,N$} \put(-15,-3){\small${
\frac{d\overline{v}_i^{\varepsilon,\delta}}{dt}=-\lambda\int_{\mathbb{R}^d}\nabla_xW^\varepsilon(\overline{x}_i^{\varepsilon,\delta}-y)\rho^{\varepsilon,\delta}(t,y)dy
-(r\overline{v}_i^{\varepsilon,\delta}+\lambda\nabla_x V(\overline{x}_i^{\varepsilon,\delta})+\beta \overline{v}_i^{\varepsilon,\delta})
+\beta\overline{u}^{\varepsilon,\delta}(\overline{x}_i^{\varepsilon,\delta})
\quad(b)}$}
\end{picture}
\end{center}\quad
\\\\\\\\\\\\
\\
\\
\\
Our paper primarily focuses on proving the first step $(a) \to (b)$, where we demonstrate that under certain conditions, the trajectories of the particle systems $(a)$ and $(b)$ are sufficiently close.

This article is organized as follows: Section 2 introduces some notations and preliminary work. Section 3 clarifies the main result and provides the corresponding proof.



\section{Notation and Preliminary work}

In order to present the analytical results in the follow section, we restrict to the following notations by rewriting \eqref{particle} and \eqref{0.6}.
Let $(X_t^{\varepsilon,\delta},V_t^{\varepsilon,\delta})$ be the trajectory on $\mathbb{R}^{2dN}$ which evolves according to the equation of motion with cut-off, i.e.,
\begin{equation}\label{def2.1-1}
\left\{
\begin{aligned}
&\frac{d}{dt}X_t^{\varepsilon,\delta}=V_t^{\varepsilon,\delta},\\
&\frac{d}{dt}V_t^{\varepsilon,\delta}=\Psi^{\varepsilon,\delta}(X_t^{\varepsilon,\delta})
+\Gamma(X_t^{\varepsilon,\delta},V_t^{\varepsilon,\delta})+\Phi^{\varepsilon,\delta}(X_t^{\varepsilon,\delta}),
\end{aligned}
\right.
\end{equation}
where $\Psi^\varepsilon(X_t^{\varepsilon,\delta})$ denotes the interaction force with
$$
(\Psi^{\varepsilon,\delta}(X_t^{\varepsilon,\delta}))_i
=-\frac{\lambda}{N}\sum_{j=1}^N\nabla_x W^\varepsilon(x_i^{\varepsilon,\delta}-x_j^{\varepsilon,\delta}),
$$
while
$$(\Gamma(X_t^{\varepsilon,\delta},V_t^{\varepsilon,\delta}))_i
=-G(x_i^{\varepsilon,\delta},v_i^{\varepsilon,\delta})
=-(rv_i^{\varepsilon,\delta}+\lambda\nabla_x V(x_i^{\varepsilon,\delta})+\beta v_i^{\varepsilon,\delta})$$
and
$$
(\Phi^{\varepsilon,\delta}(X_t^{\varepsilon,\delta}))_i
:=\beta u^{\varepsilon,\delta}(x_i^{\varepsilon,\delta}).
$$
Let $(\overline{X}_t^{\varepsilon,\delta},\overline{V}_t^{\varepsilon,\delta})$ be the trajectory on $\mathbb{R}^{2dN}$ which evolves according to the kinetic equation \eqref{0.5}, i.e.,
\begin{equation}\label{def2.2-1}
\left\{
\begin{aligned}
&\frac{d}{dt}\overline{X}_t^{\varepsilon,\delta}=\overline{V}_t^{\varepsilon,\delta},\\
&\frac{d}{dt}\overline{V}_t^{\varepsilon,\delta}=\overline{\Psi}^{\varepsilon,\delta}(\overline{X}_t^{\varepsilon,\delta})
+\Gamma(\overline{X}_t^{\varepsilon,\delta},\overline{V}_t^{\varepsilon,\delta})+\overline{\Phi}^{\varepsilon,\delta}(\overline{X}_t^{\varepsilon,\delta}),
\end{aligned}
\right.
\end{equation}
where
$(\overline{\Psi}^{\varepsilon,\delta}(\overline{X}_t^\varepsilon))_i=-\lambda\int_{\Rd}\nabla_x W^\varepsilon(\overline{x}_i^{\varepsilon,\delta}-y)\rho^{\varepsilon,\delta}(t,y)dy$ and $(\overline{\Phi}^{\varepsilon,\delta}(\overline{X}_t^{\varepsilon,\delta}))_i
=\beta\overline{u}^{\varepsilon,\delta}(\overline{x}_i^{\varepsilon,\delta})$.

In this paper we choose $(X,V)$ and $(\overline{X},\overline{V})$ represent the stochastic initial data, which are independent and identically distributed. Note we consider the same initial data for both model, that means $(X,V)=(\overline{X},\overline{V})$.

First, we point out several properties for the regularized interaction force $\nabla_x W^\varepsilon(x)$  and $G(x,v)$.
\begin{lemma}\label{lm2.1}
There are several facts for $\nabla_x W^\varepsilon(x)$ and $G(x,v)$.

\item{(i)} $\nabla_x W^\varepsilon(x)$ is bounded, i.e., $|\nabla_x W^\varepsilon(x)|\le C\varepsilon^{-(d-1)}$ .

\item{(ii)} $\nabla_x W^\varepsilon(x)$ satisfies
$$
|\nabla_x W^\varepsilon(x)-\nabla_x W^\varepsilon(y)|\le q^\varepsilon(x)|x-y|,\qquad \forall~~|x-y|<2\varepsilon,
$$
where
\begin{equation*}
q^\varepsilon(x):=\left\{
\begin{aligned}
&\frac{C}{|x|^d},&&\quad|x|\ge3\varepsilon,\\
&C\varepsilon^{-d},&&\quad|x|<3\varepsilon.
\end{aligned}
\right.
\end{equation*}
\item{(iii)} $G(x,v)$ is Lipschitz continuous, i.e.,
$$
|G(x,v)-G(x',v')|
\le L(|x-x'|+|v-v'|).
$$
\end{lemma}
\begin{proof}
 Conclusions $(i)$ and $(iii)$ can easy be achieved from the definition of  $\nabla_x W^\varepsilon(x)$ and $G(x,v)$. Thus, we omit the proof of results $(i)$ and $(iii)$.

As for $(ii)$, it follows from \cite[Lemma 6.3]{lazarovici2017mean} or \cite[Lemma 2.2]{boers2016mean}. Let $\zeta=y-x$. Indeed for $|x|<3\varepsilon$, by the definition of $\nabla_x W^\varepsilon$ it is easy to get
$$
|\nabla_x W^\varepsilon(x)-\nabla_x W^\varepsilon(y)|=|\nabla_x W^\varepsilon(x+\zeta)-\nabla_x W^\varepsilon(x)|
\le|\nabla^2_xW^\varepsilon(x')||\zeta|\le C\varepsilon^{-d}|\zeta|.
$$ For the case of $|x|\ge3\varepsilon$, there exists some with $s\in[0,1]$ such that
$$
|\nabla_x W^\varepsilon(x)-\nabla_x W^\varepsilon(y)|=|\nabla_x W^\varepsilon(x+\zeta)-\nabla_x W^\varepsilon(x)|
\le|\nabla^2_xW^\varepsilon(x+s\zeta)||\zeta|
$$
holds. One further notices that
{\small\begin{align*}
|\nabla^2_xW^\varepsilon(x+s\zeta)|\le C_d\frac{1}{|x+s\zeta|^d}\le C_d\frac{1}{\Big|x-\frac{x}{|x|}|\zeta|\Big|^d}
=C_d\frac{1}{\Big|x(1-\frac{|\zeta|}{|x|})\Big|^d}\le C_d\frac{1}{\Big|x(1-\frac{2}{3})\Big|^d}=\frac{3^dC_d}{|x|^d}.
\end{align*}}
Then we have
$$
|\nabla_x W^\varepsilon(x)-\nabla_x W^\varepsilon(y)|\le\frac{C_d}{|x|^d}|\zeta|.
$$
Therefore, with the arguments presented above, we have successfully completed the proof of $(ii)$.
\end{proof}
For the subsequent proof, we need the following assumptions.
\begin{assumption}\label{assum}
	We assume that there exists a time $t>0$ and a constant $C$ independent of $\varepsilon,\delta$ such that the solution $f^{\varepsilon,\delta}(t,x,v)$ of the kinetic equation \eqref{0.5} satisfies
	\begin{enumerate}[label=(\roman*)]
		\item
		$$
\sup_{0\le s\le t}\Big\|\int_{\Rd} |v|f^{\varepsilon,\delta}(s,\cdot,v)dv\Big\|_\infty\le C,
		$$
		\item $$
		\sup_{0\le s\le t}\Big\|\int_{\R^{2d}}|(\phi^{\varepsilon})'(x-y)|f^{\varepsilon,\delta}(s,y,v)dydv\Big\|_\infty\le C,
		$$
		\item $$
		\sup_{0\le s\le t}\Big\|\int_{\R^{2d}}\frac{1}{|x-y|^d}f^{\varepsilon,\delta}(s,y,v)dydv\Big\|_\infty\le C.
		$$
	\end{enumerate}
\end{assumption}
\begin{definition}\label{def3.1}
	Let $\theta\in(0,\frac{1}{9d+2}), \, \alpha\in\Big(\theta,\frac{1-(9d-3)\theta}{5}\Big)$ and $S_t: \mathbb{R}^{2dN}\times\mathbb{R}\rightarrow\mathbb{R}$ be the stochastic process given by
	$$
	S_t=\min\Big\{1,N^\alpha \sup_{0\le s\le t}\Big|(X_s^{\varepsilon,\delta},V_s^{\varepsilon,\delta})-(\overline{X}_s^{\varepsilon,\delta},\overline{V}_s^{\varepsilon,\delta})\Big|_\infty\Big\}.
	$$
	The set, where $|S_t|=1$, is defined as $\mathcal{N}_\alpha$, i.e.,
	\begin{align}\label{def3.1-1}
		\mathcal{N}_\alpha:=\Big\{(X,V): \sup_{0\le s\le t}\Big|(X_s^{\varepsilon,\delta},V_s^{\varepsilon,\delta})-(\overline{X}_s^{\varepsilon,\delta},\overline{V}_s^{\varepsilon,\delta})\Big|_\infty>N^{-\alpha}\Big\}.
	\end{align}
	Here and in the following we use $|\cdot|_\infty$ as the supremum norm on $\mathbb{R}^{2dN}$. Note that
	$$
	\mathbb{E}_0(S_{t+dt}-S_t|\mathcal{N}_\alpha)\le0,
	$$
	since $S_t$ takes the value of one for $(X,V)\in\mathcal{N}_\alpha$.
\end{definition}
\section{Rigorous derivation of the mean-field limit}
In this section, we give the rigorous derivation of the mean-field limit from the moderately interacting particle \eqref{particle} to the auxiliary intermediate system \eqref{0.6} in the limit of $N\to\infty$. Now we give the main result for our paper.
\begin{theorem}\label{th3.1}
Let $(X_s^{\varepsilon,\delta},V_s^{\varepsilon,\delta})$ and $(\overline{X}_s^{\varepsilon,\delta},\overline{V}_s^{\varepsilon,\delta})$ be solutions to \eqref{def2.1-1} and \eqref{def2.2-1}, respectively. Let $f^{\varepsilon,\delta}(t,x,v)$ be the distribution law of them, and it is the solution to the kinetic equation \eqref{0.5} satisfying Assumption \ref{assum}. Then there exists a constant $C$ such that
$$
\mathbb{P}_0\Big(\sup_{0\le s\le t}\Big|(X_s^{\varepsilon,\delta},V_s^{\varepsilon,\delta})-(\overline{X}_s^{\varepsilon,\delta},\overline{V}_s^{\varepsilon,\delta})\Big|_\infty>N^{-\alpha}\Big)
\le C \exp\Big\{\Big(C+C\vartheta\ln(N)\Big)t\Big\}\cdot N^{-n},
$$
where $\varepsilon=N^{-\theta},\,\theta\in(0,\frac{1}{9d+2}),\,
\alpha\in\Big(\theta,\frac{1-(9d-3)\theta}{5}\Big),\,
\kappa\in\Big((d-1)\theta+\alpha,\frac{1-5(d-1)\theta-\alpha}{4}\Big),\,
\gamma\in\Big(0,\frac{1-(3d-1)\theta-\alpha}{4}\Big),\,
\eta\in\Big((d-1)\theta+\alpha,\frac{1-(5d+1)\theta-\alpha}{4}\Big),\,
\mu\in\Big(0,\frac{1-(5d+3)\theta}{4}\Big),\,n=\min\{1-5(d-1)\theta-4\kappa-\alpha,\,\,1-(3d-1)\theta-4\gamma-\alpha,\,\,1-(5d+1)\theta-4\eta-\alpha,\,\,1-(5d+3)\theta-4\mu-\alpha,\,\,\kappa-\alpha-(d-1)\theta,\,\,\eta-\alpha-(d-1)\theta\},\,\,\delta=\frac{1}{\sqrt{\vartheta \ln N}}$, and $\vartheta\in\Big(0,\min\{\frac{n}{Ct},\theta\}\Big)$.
\end{theorem}
First, let's define some sets as follows.
\begin{definition}\label{def3.2}
The sets $\mathcal{N}_\kappa,\,\mathcal{N}_\gamma,\,\mathcal{N}_\eta$ and $\mathcal{N}_\mu$ are characterized by
\begin{align}\label{def3.2-1}
\mathcal{N}_\kappa:=\Big\{(X,V):\Big|\Psi^{\varepsilon,\delta}(\overline{X}_t^{\varepsilon,\delta})
-\overline{\Psi}^{\varepsilon,\delta}(\overline{X}_t^{\varepsilon,\delta})\Big|_\infty>N^{-\kappa}\Big\},
\end{align}
\begin{align}\label{def3.2-2}
\mathcal{N}_\gamma:=\Big\{(X,V):\Big|Q^{\varepsilon,\delta}(\overline{X}_t^{\varepsilon,\delta})
-\overline{Q}^{\varepsilon,\delta}(\overline{X}_t^{\varepsilon,\delta})\Big|_\infty>N^{-\gamma}\Big\},
\end{align}
\begin{align}\label{def3.2-3}
\mathcal{N}_\eta:=\Big\{(X,V):\Big|\Phi^{\varepsilon,\delta}(\overline{X}_t^{\varepsilon,\delta})
-\overline{\Phi}^{\varepsilon,\delta}(\overline{X}_t^{\varepsilon,\delta})\Big|_\infty>N^{-\eta}\Big\},
\end{align}
\begin{align}\label{3.2-4}
\mathcal{N}_\mu:=\Big\{(X,V):\Big|P^{\varepsilon,\delta}(\overline{X}_t^{\varepsilon,\delta})
-\overline{P}^{\varepsilon,\delta}(\overline{X}_t^{\varepsilon,\delta})\Big|_\infty>N^{-\mu}\Big\},
\end{align}
where $Q^{\varepsilon,\delta}(\overline{X}_t^{\varepsilon,\delta})$ and $\overline{Q}^{\varepsilon,\delta}(\overline{X}_t^{\varepsilon,\delta})$ are understood in the sense of
$$
(Q^{\varepsilon,\delta}(\overline{X}_t^{\varepsilon,\delta}))_i:=-\frac{\lambda}{N}
\sum_{j=1}^Nq^\varepsilon(\overline{x}_i^{\varepsilon,\delta}-\overline{x}_j^{\varepsilon,\delta})
$$
and correspondingly
$$
(\overline{Q}^{\varepsilon,\delta}(\overline{X}_t^{\varepsilon,\delta}))_i:=-\lambda\int_{\mathbb{R}^d} q^\varepsilon(\overline{x}_i^{\varepsilon,\delta}-y)\rho^{\varepsilon,\delta}(t,y)dy.
$$
Similarly $P^{\varepsilon,\delta}(\overline{X}_t^{\varepsilon,\delta})$ and $\overline{P}^{\varepsilon,\delta}(\overline{X}_t^{\varepsilon,\delta})$ are understood in the sense of
$$
(P^{\varepsilon,\delta}(\overline{X}_t^{\varepsilon,\delta}))_i:=\frac{1}{N}
\sum_{j=1}^N|(\phi^{\varepsilon})'(\overline{x}_i^{\varepsilon,\delta}-\overline{x}_j^{\varepsilon,\delta})|
$$
and correspondingly
$$
(\overline{P}^{\varepsilon,\delta}(\overline{X}_t^{\varepsilon,\delta}))_i:=\int_{\mathbb{R}^d}|(\phi^{\varepsilon})'(\overline{x}_i^{\varepsilon,\delta}-y)|\rho^{\varepsilon,\delta}(t,y)dy.
$$
\end{definition}
Next, we will prove that the probability of the sets $\mathcal{N}_\kappa$, $\mathcal{N}_\gamma$, $\mathcal{N}_\eta$ and $\mathcal{N}_\mu$ tend to $0$ as $N$ goes to infinity by the following Lemmas \ref{lm3.1}-\ref{lm3.3'}.
\begin{lemma}\label{lm3.1}
There exists a constant $C>0$ such that
$$
\mathbb{P}_0(\mathcal{N}_\kappa)\le \lambda^4C\varepsilon^{-4(d-1)}N^{-(1-4\kappa)}.
$$
\end{lemma}
\begin{proof}
First, we let the set $\mathcal{N}_\kappa$ evolve along the characteristics of the kinetic equation
$$
\mathcal{N}_{\kappa,t}:=\Big\{(\overline{X}_t^{\varepsilon,\delta},\overline{V}_t^{\varepsilon,\delta}):
\Big|N^\kappa\Psi^{\varepsilon,\delta}(\overline{X}_t^{\varepsilon,\delta})-N^\kappa\overline{\Psi}^{\varepsilon,\delta}(\overline{X}_t^{\varepsilon,\delta})\Big|_\infty>1\Big\}
$$
and consider the following fact
$$
\mathcal{N}_{\kappa,t}\subseteq\bigoplus_{i=1}^N\mathcal{N}_{\kappa,t}^i,
$$
where
$$
\mathcal{N}_{\kappa,t}^i:=\Big\{(\overline{x}_i^{\varepsilon,\delta},\overline{v}_i^{\varepsilon,\delta}):
\Big|N^\kappa\cdot\frac{\lambda}{N}\sum_{j=1}^N
\nabla_x W^\varepsilon(\overline{x}_i^{\varepsilon,\delta}-\overline{x}_j^{\varepsilon,\delta})
-\lambda N^\kappa(\nabla_x W^\varepsilon\ast\rho^{\varepsilon,\delta})(t,\overline{x}_i^{\varepsilon,\delta})\Big|_\infty>1\Big\}.
$$
So, using the symmetry argument in exchanging any two coordinates, we can get
$$
\mathbb{P}_t(\mathcal{N}_{\kappa,t})\le\sum_{i=1}^N\mathbb{P}_t(\mathcal{N}_{\kappa,t}^i)=N\mathbb{P}_t(\mathcal{N}_{\kappa,t}^1).
$$
Using Markov inequality gives
\begin{align}\label{lm3.1-1'}
\mathbb{P}_t(\mathcal{N}_{\kappa,t}^1)
&\le \mathbb{E}_t\Big[\Big(N^\kappa\cdot\frac{\lambda}{N}\sum_{j=1}^N
\nabla_x W^\varepsilon(\overline{x}_1^{\varepsilon,\delta}-\overline{x}_j^{\varepsilon,\delta})
-\lambda N^\kappa(\nabla_x W^\varepsilon\ast\rho^{\varepsilon,\delta})(t,\overline{x}_1^{\varepsilon,\delta})\Big)^4\Big]\nonumber\\
&=\Big(\frac{\lambda N^\kappa}{N}\Big)^4
\mathbb{E}_t\Big[\Big(\sum_{j=1}^N
\nabla_x W^\varepsilon(\overline{x}_1^{\varepsilon,\delta}-\overline{x}_j^{\varepsilon,\delta})
-N(\nabla_x W^\varepsilon\ast\rho^{\varepsilon,\delta})(t,\overline{x}_1^{\varepsilon,\delta})\Big)^4\Big].
\end{align}
Let $h_j:=\nabla_x W^\varepsilon(\overline{x}_1^{\varepsilon,\delta}-\overline{x}_j^{\varepsilon,\delta})-\int_{\mathbb{R}^d}\nabla_x W^\varepsilon(\overline{x}_1^{\varepsilon,\delta}-y)\rho^{\varepsilon,\delta}(t,y)dy$. Then, the each term in the expectation \eqref{lm3.1-1'} takes the form of $\prod\limits_{j=1}^Nh_j^{k_j}$ with $\sum\limits_{j=1}^Nk_j=4$, and the expectation assumes the value of zero whenever there exists a $j$ such that $k_j=1$, i.e.,
$$
\mathbb{E}_t\Big[\nabla_x W^\varepsilon(\overline{x}_1^{\varepsilon,\delta}-\overline{x}_j^{\varepsilon,\delta})-\int_{\mathbb{R}^d}\nabla_x W^\varepsilon(\overline{x}_1^{\varepsilon,\delta}-y)\rho^{\varepsilon,\delta}(t,y)dy\Big]=0.
$$
Then, we simplify the estimate \eqref{lm3.1-1'} to
$$
\mathbb{P}_t(\mathcal{N}_{\kappa,t}^1)
\le\Big(\frac{\lambda N^\kappa}{N}\Big)^4
\mathbb{E}_t\Big[\sum_{j=1}^Nh_j^4+\sum_{1\le m<n}^NC_4^2h_m^2h_n^2\Big].
$$
Since $\nabla_x W^\varepsilon$ and $\|\rho^{\varepsilon,\delta}\|_{L^1}$ are bounded, thus for any fixed $j$
$$
|h_j|\le\Big|\nabla_xW^{\varepsilon}(\overline{x}_1^{\varepsilon,\delta}-\overline{x}_j^{\varepsilon,\delta})\Big|
+\int_{\mathbb{R}^d}\Big|\nabla_xW^{\varepsilon}(\overline{x}_1^{\varepsilon,\delta}-y)\Big|\rho^{\varepsilon,\delta}(t,y)dy
\le C\varepsilon^{-(d-1)}.
$$
Therefore $|h_j|$ is bounded to any power and we get
$$
\mathbb{E}_t\Big[h_m^2h_n^2\Big]\le C\varepsilon^{-4(d-1)}\quad \quad
\mathbb{E}_t\Big[h_j^4\Big]\le C\varepsilon^{-4(d-1)},
$$
and consequently
\begin{align*}
\mathbb{P}_t(\mathcal{N}_{\kappa,t}^1)
&\le C\varepsilon^{-4(d-1)}\Big(\frac{\lambda N^\kappa}{N}\Big)^4
\Big(N+\frac{N(N-1)}{2}\Big)
\le \lambda^4C\varepsilon^{-4(d-1)}N^{-(2-4\kappa)}.
\end{align*}
Then, we obtain
$$
\mathbb{P}_0(\mathcal{N}_\kappa)
=\mathbb{P}_t(\mathcal{N}_{\kappa,t})
\le N\mathbb{P}_t(\mathcal{N}_{\kappa,t}^1)
\le N\lambda^4C\varepsilon^{-4(d-1)}N^{-(2-4\kappa)}=\lambda^4C\varepsilon^{-4(d-1)}N^{-(1-4\kappa)}.
$$
\end{proof}
\begin{lemma}\label{lm3.2}
There exists a constant $C>0$ such that
$$
\mathbb{P}_0(\mathcal{N}_\gamma)\le \lambda^4C\varepsilon^{-2d}N^{-(1-4\gamma)}.
$$
\end{lemma}
\begin{proof}
Let the set $\mathcal{N}_\gamma$ evolve along the characteristics of the kinetic equation
$$
\mathcal{N}_{\gamma,t}:=\Big\{(\overline{X}_t^{\varepsilon,\delta},\overline{V}_t^{\varepsilon,\delta}):
\Big|N^\gamma Q^{\varepsilon,\delta}(\overline{X}_t^{\varepsilon,\delta},\overline{V}_t^{\varepsilon,\delta})-N^\gamma\overline{Q}^{\varepsilon,\delta}(\overline{X}_t^{\varepsilon,\delta},\overline{V}_t^{\varepsilon,\delta})\Big|_\infty>1\Big\}
$$
and consider the following fact
$$
\mathcal{N}_{\gamma,t}\subseteq\bigoplus_{i=1}^N\mathcal{N}_{\gamma,t}^i,
$$
where
$$
\mathcal{N}_{\gamma,t}^i
:=\Big\{(\overline{x}_i^{\varepsilon,\delta},\overline{v}_i^{\varepsilon,\delta}):\Big|N^\gamma\cdot\frac{\lambda}{N}
\sum_{j=1}^N
q^\varepsilon(\overline{x}_i^{\varepsilon,\delta}-\overline{x}_j^{\varepsilon,\delta})
-\lambda N^\gamma(q^{\varepsilon}\ast\rho^{\varepsilon,\delta})(t,\overline{x}_i^{\varepsilon,\delta})\Big|_\infty>1\Big\}.
$$
Due to the symmetry in exchanging any two coordinates, we get
$$
\mathbb{P}_t(\mathcal{N}_{\gamma,t})\le\sum_{i=1}^N\mathbb{P}_t(\mathcal{N}_{\gamma,t}^i)=N\mathbb{P}_t(\mathcal{N}_{\gamma,t}^1).
$$
Using Markov inequality gives
\begin{align}\label{lm3.2-1}
\mathbb{P}_t(\mathcal{N}_{\gamma,t}^i)
&\le \mathbb{E}_t\Big[\Big(N^\gamma\cdot\frac{\lambda}{N}
\sum_{j=1}^N
q^{\varepsilon}(\overline{x}_1^{\varepsilon,\delta}-\overline{x}_j^{\varepsilon,\delta})
-\lambda N^\gamma(q^{\varepsilon}\ast\rho^{\varepsilon,\delta})(t,\overline{x}_1^{\varepsilon,\delta})\Big)^4\Big]\nonumber\\
&=\Big(\frac{\lambda N^\gamma}{N}\Big)^4
\mathbb{E}_t\Big[\Big(
\sum_{j=1}^N
q^{\varepsilon}(\overline{x}_1^{\varepsilon,\delta}-\overline{x}_j^{\varepsilon,\delta})
-N(q^{\varepsilon}\ast\rho^{\varepsilon,\delta})(t,\overline{x}_1^{\varepsilon,\delta})\Big)^4\Big].
\end{align}
Similar to Lemma \ref{lm3.1}, we define $h_j:=q^{\varepsilon}(\overline{x}_1^{\varepsilon,\delta}-\overline{x}_j^{\varepsilon,\delta})
-\int_{\mathbb{R}^d} q^{\varepsilon}(\overline{x}_1^{\varepsilon,\delta}-y)
\rho^{\varepsilon,\delta}(t,y)dy$. With the same argument as in Lemma \ref{lm3.1}, we get
$$
\mathbb{P}_t(\mathcal{N}_{\gamma,t}^1)
\le\Big(\frac{\lambda N^\gamma}{N}\Big)^4
\mathbb{E}_t\Big[\sum_{j=1}^Nh_j^4+\sum_{1\le m<n}^N6h_m^2h_n^2\Big].
$$
On the other hand, due to the cut-off, we have
$$
\|q^{\varepsilon}\|_\infty\le C\varepsilon^{-d}.
$$
Using the $L^\infty$-norm of $q^{\varepsilon}$ and the integrability of $\rho^{\varepsilon,\delta}$, we obtain
$$
\Big|\int_{\mathbb{R}^d} q^{\varepsilon}(\overline{x}_1^{\varepsilon,\delta}-y)
\rho^{\varepsilon,\delta}(t,y)dy\Big|\le C\varepsilon^{-d}.
$$
Then
$|h_j|\le C\varepsilon^{-d}$.
Furthermore, by applying the inequality $\mathbb{E}\Big[(X-\mathbb{E}[X])^2\Big]\le\mathbb{E}[X^2]$ for any random variable $X$, we have for any fixed $j$
\begin{align*}
\mathbb{E}_t[h_j^2]&=\mathbb{E}_t\Big[\Big(q^{\varepsilon}(\overline{x}_1^{\varepsilon,\delta}-\overline{x}_j^{\varepsilon,\delta})
-\int_{\mathbb{R}^d} q^{\varepsilon}(\overline{x}_1^{\varepsilon,\delta}-y)
\rho^{\varepsilon,\delta}(t,y)dy\Big)^2\Big]\\
&\le\mathbb{E}_t\Big[\Big(q^{\varepsilon}(\overline{x}_1^{\varepsilon,\delta}-\overline{x}_j^{\varepsilon,\delta})\Big)^2\Big]\\
&=\int_{\R^{2d}}\Big(\int_{\R^{2d}}(q^{\varepsilon}(x-y))^2f^{\varepsilon,\delta}(t,y,v)dydv\Big)f^{\varepsilon,\delta}(t,x,w)dxdw.
\end{align*}
Notice that
\begin{align}\label{3.2-5}
\int_{\R^{2d}}(q^{\varepsilon}(x-y))^2f^{\varepsilon,\delta}(t,y,v)dydv
&\le\int_{\{|x-y|<3\varepsilon\}\times\Rd}(C\varepsilon^{-d})^2f^{\varepsilon,\delta}(t,y,v)dydv\nonumber\\
&~~~+\int_{\{|x-y|\ge3\varepsilon\}\times\Rd}(C\frac{1}{|x-y|^d})^2f^{\varepsilon,\delta}(t,y,v)dydv\nonumber\\
&=:I_1+I_2.
\end{align}
Applying $(iii)$ of  Assumption \ref{assum}, we get that
\begin{align*}
I_1&=C\varepsilon^{-d}\int_{\{|x-y|<3\varepsilon\}\times\Rd}\varepsilon^{-d}f^{\varepsilon,\delta}(t,y,v)dydv\\
&\le C\varepsilon^{-d}\int_{\{|x-y|<3\varepsilon\}\times\Rd}\frac{1}{|x-y|^d}f^{\varepsilon,\delta}(t,y,v)dydv,
\end{align*}
and
\begin{align*}
I_2&=C\int_{\{|x-y|\ge3\varepsilon\}\times\Rd}\frac{1}{|x-y|^d}\frac{1}{|x-y|^d}f^{\varepsilon,\delta}(t,y,v)dydv\\
&\le C\varepsilon^{-d}\int_{\{|x-y|\ge3\varepsilon\}\times\Rd}\frac{1}{|x-y|^d}f^{\varepsilon,\delta}(t,y,v)dydv.
\end{align*}
Taking the above two term into \eqref{3.2-5}, we get
\begin{align*}
\int_{\R^{2d}}(q^{\varepsilon}(x-y))^2f^{\varepsilon,\delta}(t,y,v)dydv
&\le C\varepsilon^{-d}\int_{\R^{2d}}\frac{1}{|x-y|^d}f^{\varepsilon,\delta}(t,y,v)dydv
\le C\varepsilon^{-d}.
\end{align*}
So, we obtain
\begin{align*}
\mathbb{E}_t[h_j^2]\le C\varepsilon^{-d}\int_{\R^{2d}}f^{\varepsilon,\delta}(t,x,w)dxdw
\le C\varepsilon^{-d}.
\end{align*}
Therefore for any fixed $j$
$$
\mathbb{E}_t[h_j^4]\le\|h_j\|_\infty^2\mathbb{E}_t[h_j^2]\le C\varepsilon^{-2d}\cdot\frac{2\pi}{d}\varepsilon^{-d}
=C\varepsilon^{-3d}\quad\quad\quad\mathbb{E}_t[h_m^2h_n^2]\le C\varepsilon^{-2d}.
$$
Consequently
$$
\mathbb{P}_t(\mathcal{N}_{\gamma,t}^1)
\le C\Big(\frac{\lambda N^\gamma}{N}\Big)^4
\Big(\varepsilon^{-3d}N+\varepsilon^{-2d}\frac{N(N-1)}{2}\Big)
\le \lambda^4C\varepsilon^{-2d}N^{-(2-4\gamma)}.
$$
Then
$$
\mathbb{P}_0(\mathcal{N}_\gamma)
=\mathbb{P}_t(\mathcal{N}_{\gamma,t})
\le N\mathbb{P}_t(\mathcal{N}_{\gamma,t}^1)
\le N\lambda^4C\varepsilon^{-2d}N^{-(2-4\gamma)}=\lambda^4C\varepsilon^{-2d}N^{-(1-4\gamma)}.
$$
\end{proof}
\begin{lemma}\label{lm3.3}
There exists a constant $C>0$ such that
$$
\mathbb{P}_0(\mathcal{N}_\eta)\le C\beta^4\varepsilon^{-4d}N^{-(1-4\eta)}(R^4+\delta^{-4}).
$$
\end{lemma}
\begin{proof}
First, we let the set $\mathcal{N}_\eta$ evolve along the characteristics of the kinetic equation
$$
\mathcal{N}_{\eta,t}:=\Big\{(\overline{X}_t^{\varepsilon,\delta},\overline{V}_t^{\varepsilon,\delta}):
\Big|N^\eta\Phi^{\varepsilon,\delta}(\overline{X}_t^{\varepsilon,\delta})-N^\eta\overline{\Phi}^{\varepsilon,\delta}(\overline{X}_t^{\varepsilon,\delta})\Big|_\infty>1\Big\}
$$
and consider the following fact
$$
\mathcal{N}_{\eta,t}\subseteq\bigoplus_{i=1}^N\mathcal{N}_{\eta,t}^i,
$$
where
{\footnotesize\begin{align*}
\mathcal{N}_{\eta,t}^i:=\left\{(\overline{x}_i^{\varepsilon,\delta},\overline{v}_i^{\varepsilon,\delta}):
\Big|N^\eta\beta\frac{\frac{1}{N}\sum\limits_{j=1}^N \phi_R(\overline{v}_j^{\varepsilon,\delta})\phi^{\varepsilon}(\overline{x}_i^{\varepsilon,\delta}-\overline{x}_j^{\varepsilon,\delta})}
{\frac{1}{N}\sum\limits_{j=1}^N\phi^{\varepsilon}(\overline{x}_i^{\varepsilon,\delta}-\overline{x}_j^{\varepsilon,\delta})+\delta}
-N^\eta
\beta\frac{\int_{\R^{2d}}v\phi^{\varepsilon}(\overline{x}_i^{\varepsilon,\delta}-y) f(t,y,v)dvdy}{\int_{\R^{2d}}\phi^{\varepsilon}(\overline{x}_i^{\varepsilon,\delta}-y) f(t,y,v)dvdy+\delta}
\Big|_\infty>1\right\}.
\end{align*}}
So, using the symmetry argument in exchanging any two coordinates, we can get
$$
\mathbb{P}_t(\mathcal{N}_{\eta,t})\le\sum_{i=1}^N\mathbb{P}_t(\mathcal{N}_{\eta,t}^i)=N\mathbb{P}_t(\mathcal{N}_{\eta,t}^1).
$$
Using Markov inequality gives
\begin{align*}
\mathbb{P}_t(\mathcal{N}_{\eta,t}^1)
&\le \mathbb{E}_t\Big[
\Big(N^\eta\beta\frac{\frac{1}{N}\sum\limits_{j=1}^N \phi_R(\overline{v}_j^{\varepsilon,\delta})\phi^{\varepsilon}(\overline{x}_1^{\varepsilon,\delta}-\overline{x}_j^{\varepsilon,\delta})}
{\frac{1}{N}\sum\limits_{j=1}^N\phi^{\varepsilon}(\overline{x}_1^{\varepsilon,\delta}-\overline{x}_j^{\varepsilon,\delta})+\delta}
-N^\eta
\beta\frac{\int_{\R^{2d}}v\phi^{\varepsilon}(\overline{x}_1^{\varepsilon,\delta}-y) f(t,y,v)dvdy}{\int_{\R^{2d}}\phi^{\varepsilon}(\overline{x}_1^{\varepsilon,\delta}-y) f(t,y,v)dvdy+\delta}
\Big)^4\Big]\nonumber\\
&\le C(N^\eta\beta)^4
\mathbb{E}_t\Big[\Big\{\frac{\Big(\frac{1}{N}\sum\limits_{j=1}^N \phi_R(\overline{v}_j^{\varepsilon,\delta})\phi^{\varepsilon}(\overline{x}_1^{\varepsilon,\delta}-\overline{x}_j^{\varepsilon,\delta})-\int_{\R^{2d}}v\phi^{\varepsilon}(\overline{x}_1^{\varepsilon,\delta}-y) f(t,y,v)dvdy\Big)}{\frac{1}{N}\sum\limits_{j=1}^N\phi^{\varepsilon}(\overline{x}_1^{\varepsilon,\delta}-\overline{x}_j^{\varepsilon,\delta})+\delta}\Big\}^4\Big]\nonumber\\
&~~~+C(N^\eta\beta)^4
\mathbb{E}_t\Big[\Big\{\int_{\R^{2d}}v\phi^{\varepsilon}(\overline{x}_1^{\varepsilon,\delta}-y) f(t,y,v)dvdy\nonumber\\
&\qquad\qquad\qquad\qquad\quad\Big(\int_{\R^{2d}}\phi^{\varepsilon}(\overline{x}_1^{\varepsilon,\delta}-y) f(t,y,v)dvdy-\frac{1}{N}\sum\limits_{j=1}^N \phi^{\varepsilon}(\overline{x}_1^{\varepsilon,\delta}-\overline{x}_j^{\varepsilon,\delta})\Big)\nonumber\\
&\qquad\qquad\qquad\qquad\Bigg/\Big(\frac{1}{N}\sum\limits_{j=1}^N\phi^{\varepsilon}(\overline{x}_1^{\varepsilon,\delta}-\overline{x}_j^{\varepsilon,\delta})
+\delta\Big)\Big(\int_{\R^{2d}}\phi^{\varepsilon}(\overline{x}_1^{\varepsilon,\delta}-y) f(t,y,v)dvdy+\delta\Big)\Big\}^4\Big]\nonumber\\
&\le C(\frac{N^\eta\beta}{\delta})^4\mathbb{E}_t\Big[\Big(\frac{1}{N}\sum\limits_{j=1}^N \phi_R(\overline{v}_j^{\varepsilon,\delta})\phi^{\varepsilon}(\overline{x}_1^{\varepsilon,\delta}-\overline{x}_j^{\varepsilon,\delta})-\int_{\R^{2d}}v\phi^{\varepsilon}(\overline{x}_1^{\varepsilon,\delta}-y) f(t,y,v)dvdy\Big)^4\Big]\nonumber\\
&~~~+C(\frac{N^\eta\beta}{\delta^2})^4
\mathbb{E}_t\Big[\Big(\int_{\R^{2d}}v\phi^{\varepsilon}(\overline{x}_1^{\varepsilon,\delta}-y) f(t,y,v)dvdy\nonumber\\
&\qquad\qquad\qquad\qquad\qquad(\int_{\R^{2d}}\phi^{\varepsilon}(\overline{x}_1^{\varepsilon,\delta}-y) f(t,y,v)dvdy-\frac{1}{N}\sum\limits_{j=1}^N \phi^{\varepsilon}(\overline{x}_1^{\varepsilon,\delta}-\overline{x}_j^{\varepsilon,\delta}))\Big)^4\Big]\nonumber\\
&\le C\Big(\frac{N^\eta\beta}{\delta N}\Big)^4\mathbb{E}_t\Big[\Big(\sum\limits_{j=1}^N \phi_R(\overline{v}_j^{\varepsilon,\delta})\phi^{\varepsilon}(\overline{x}_1^{\varepsilon,\delta}-\overline{x}_j^{\varepsilon,\delta})-N\int_{\R^{2d}}v\phi^{\varepsilon}(\overline{x}_1^{\varepsilon,\delta}-y) f(t,y,v)dvdy\Big)^4\Big]\nonumber\\
&~~~+C\Big(\frac{N^\eta\beta}{\delta^2 N}\Big)^4\mathbb{E}_t\Big[\Big(\int_{\R^{2d}}v\phi^{\varepsilon}(\overline{x}_1^{\varepsilon,\delta}-y) f(t,y,v)dvdy\nonumber\\
&\quad\quad\quad\quad\quad\quad\quad\quad\quad\quad(N\int_{\R^{2d}}\phi^{\varepsilon}(\overline{x}_1^{\varepsilon,\delta}-y) f(t,y,v)dvdy
-\sum\limits_{j=1}^N \phi^{\varepsilon}(\overline{x}_1^{\varepsilon,\delta}-\overline{x}_j^{\varepsilon,\delta}))\Big)^4\Big]\nonumber\\
&=:I_1+I_2.
\end{align*}

First, we consider the term $I_1$. Similar to Lemma \ref{lm3.1}, we define $h_j:=\phi_R(\overline{v}_j^{\varepsilon,\delta})\phi^{\varepsilon}(\overline{x}_1^{\varepsilon,\delta}-\overline{x}_j^{\varepsilon,\delta})-\int_{\mathbb{R}^d}\int_{\mathbb{R}^d}v\phi^{\varepsilon}(\overline{x}_1^{\varepsilon,\delta}-y) f(t,y,v)dvdy$. With the same argument as in Lemma \ref{lm3.1}, we get
$$
I_1=C\Big(\frac{N^\eta\beta}{\delta N}\Big)^4
\mathbb{E}_t\Big[\sum_{j=1}^Nh_j^4+\sum_{1\le m<n}^N6h_m^2h_n^2\Big].
$$
Notice that
\begin{align*}
\mathbb{E}_t\Big[h_j^4\Big]
&=\mathbb{E}_t\Big[\Big(\phi_R(\overline{v}_j^{\varepsilon,\delta})\phi^{\varepsilon}(\overline{x}_1^{\varepsilon,\delta}-\overline{x}_j^{\varepsilon,\delta})-\int_{\R^{2d}}v\phi^{\varepsilon}(\overline{x}_1^{\varepsilon,\delta}-y) f(t,y,v)dvdy\Big)^4\Big]\\
&\le\mathbb{E}_t\Big[\Big(|\phi_R(\overline{v}_j^{\varepsilon,\delta})\phi^{\varepsilon}(\overline{x}_1^{\varepsilon,\delta}-\overline{x}_j^{\varepsilon,\delta})|
+|\int_{\R^{2d}}v\phi^{\varepsilon}(\overline{x}_1^{\varepsilon,\delta}-y) f(t,y,v)dvdy|\Big)^4\Big]\\
&\le \varepsilon^{-4d}\mathbb{E}_t\Big[(|\phi_R(\overline{v}_j^{\varepsilon,\delta})|+C)^4\Big]\\
&\le CR^4\varepsilon^{-4d},
\end{align*}
and
\begin{align*}
\mathbb{E}_t\Big[h_m^2h_n^2\Big]
&=\mathbb{E}_t\Big[\Big(\phi_R(\overline{v}_m^{\varepsilon,\delta})\phi^{\varepsilon}(\overline{x}_1^{\varepsilon,\delta}-\overline{x}_m^{\varepsilon,\delta})-\int_{\R^{2d}}v\phi^{\varepsilon}(\overline{x}_1^{\varepsilon,\delta}-y) f(t,y,v)dvdy\Big)^2\\
&~~~~~~~~\Big(\phi_R(\overline{v}_n^{\varepsilon,\delta})\phi^{\varepsilon}(\overline{x}_1^{\varepsilon,\delta}-\overline{x}_n^{\varepsilon,\delta})-\int_{\R^{2d}}v\phi^{\varepsilon}(\overline{x}_1^{\varepsilon,\delta}-y) f(t,y,v)dvdy\Big)^2\Big]\\
&\le\mathbb{E}_t\Big[\Big(|\phi_R(\overline{v}_m^{\varepsilon,\delta})\phi^{\varepsilon}(\overline{x}_1^{\varepsilon,\delta}-\overline{x}_m^{\varepsilon,\delta})|
+|\int_{\R^{2d}}v\phi^{\varepsilon}(\overline{x}_1^{\varepsilon,\delta}-y) f(t,y,v)dvdy|\Big)^2\\
&~~~~~~~~\Big(|\phi_R(\overline{v}_n^{\varepsilon,\delta})\phi^{\varepsilon}(\overline{x}_1^{\varepsilon,\delta}-\overline{x}_n^{\varepsilon,\delta})|
+|\int_{\R^{2d}}v\phi^{\varepsilon}(\overline{x}_1^{\varepsilon,\delta}-y) f(t,y,v)dvdy|\Big)^2\Big]\\
&\le\varepsilon^{-4d}\mathbb{E}_t\Big[(|\phi_R(\overline{v}_m^{\varepsilon,\delta})|+C)^2\Big]
\mathbb{E}_t\Big[(|\phi_R(\overline{v}_n^{\varepsilon,\delta})|+C)^2\Big]\\
&\le CR^4\varepsilon^{-4d}.
\end{align*}
Then we have
\begin{align*}
I_1&=C\Big(\frac{N^\eta\beta}{\delta N}\Big)^4
\mathbb{E}_t\Big[\sum_{j=1}^Nh_j^4+\sum_{1\le m<n}^N6h_m^2h_n^2\Big]\\
&\le CR^4\varepsilon^{-4d}\Big(\frac{N^\eta\beta}{\delta N}\Big)^4
\Big(N+\frac{N(N-1)}{2}\Big)\le C\beta^4R^4\varepsilon^{-4d}\delta^{-4}N^{-(2-4\eta)}.
\end{align*}

Next, we deal with the term $I_2$. Similar to $I_1$, we get
\begin{align*}
I_2\le C\varepsilon^{-4d}\Big(\frac{N^\eta\beta}{\delta^2 N}\Big)^4
\Big(N+\frac{N(N-1)}{2}\Big)\le C\beta^4\varepsilon^{-4d}\delta^{-8}N^{-(2-4\eta)}.
\end{align*}
Therefore, we have
$$
\mathbb{P}_t(\mathcal{N}_{\eta,t}^1)\le C\beta^4\varepsilon^{-4d}\delta^{-4}N^{-(2-4\eta)}(R^4+\delta^{-4}),
$$
then
$$
\mathbb{P}_0(\mathcal{N}_\eta)
=\mathbb{P}_t(\mathcal{N}_{\eta,t})
\le N\mathbb{P}_t(\mathcal{N}_{\eta,t}^1)
\le C\beta^4\varepsilon^{-4d}N^{-(1-4\eta)}(R^4+\delta^{-4}).
$$
\end{proof}
\begin{lemma}\label{lm3.3'}
There exists a constant $C>0$ such that
$$
\mathbb{P}_0(\mathcal{N}_\mu)\le C\varepsilon^{-4(d+1)}N^{-(1-4\mu)}.
$$
\end{lemma}
\begin{proof}
First, we let the set $\mathcal{N}_\mu$ evolve along the characteristics of the kinetic equation
$$
\mathcal{N}_{\mu,t}:=\Big\{(\overline{X}_t^{\varepsilon,\delta},\overline{V}_t^{\varepsilon,\delta}):
\Big|N^\mu P^{\varepsilon,\delta}(\overline{X}_t^{\varepsilon,\delta})-N^\mu\overline{P}^{\varepsilon,\delta}(\overline{X}_t^{\varepsilon,\delta})\Big|_\infty>1\Big\}
$$
and consider the following fact
$$
\mathcal{N}_{\mu,t}\subseteq\bigoplus_{i=1}^N\mathcal{N}_{\mu,t}^i,
$$
where
{\small\begin{align*}
\mathcal{N}_{\mu,t}^i:=\left\{(\overline{x}_i^{\varepsilon,\delta},\overline{v}_i^{\varepsilon,\delta}):
\Big|N^\mu\frac{1}{N}
\sum_{j=1}^N|\phi^{\varepsilon'}(\overline{x}_i^{\varepsilon,\delta}-\overline{x}_j^{\varepsilon,\delta})|
-N^\mu\int_{\mathbb{R}^d}|\phi^{\varepsilon'}(\overline{x}_i^{\varepsilon,\delta}-y)|\rho^{\varepsilon,\delta}(t,y)dy\Big|_\infty>1
\right\}.
\end{align*}}
So, using the symmetry argument in exchanging any two coordinates, we can get
$$
\mathbb{P}_t(\mathcal{N}_{\mu,t})\le\sum_{i=1}^N\mathbb{P}_t(\mathcal{N}_{\mu,t}^i)=N\mathbb{P}_t(\mathcal{N}_{\mu,t}^1).
$$
Using Markov inequality gives
\begin{align*}
\mathbb{P}_t(\mathcal{N}_{\mu,t}^1)
&\le\mathbb{E}_t\Big[\Big(N^\mu\frac{1}{N}
\sum_{j=1}^N|\phi^{\varepsilon'}(\overline{x}_1^{\varepsilon,\delta}-\overline{x}_j^{\varepsilon,\delta})|
-N^\mu\int_{\mathbb{R}^d}|\phi^{\varepsilon'}(\overline{x}_1^{\varepsilon,\delta}-y)|\rho^{\varepsilon,\delta}(t,y)dy\Big)^4\Big]\\
&\le\Big(\frac{N^\mu}{N}\Big)^4\mathbb{E}_t\Big[\Big(\sum_{j=1}^N|\phi^{\varepsilon'}(\overline{x}_1^{\varepsilon,\delta}-\overline{x}_j^{\varepsilon,\delta})|
-N\int_{\mathbb{R}^d}|\phi^{\varepsilon'}(\overline{x}_1^{\varepsilon,\delta}-y)|\rho^{\varepsilon,\delta}(t,y)dy\Big)^4\Big].
\end{align*}
Define $h_j:=|\phi^{\varepsilon'}(\overline{x}_1^{\varepsilon,\delta}-\overline{x}_j^{\varepsilon,\delta})|
-\int_{\mathbb{R}^d}|\phi^{\varepsilon'}(\overline{x}_1^{\varepsilon,\delta}-y)|\rho^{\varepsilon,\delta}(t,y)dy$. Similar to Lemma \ref{lm3.1}, we get
$$
\mathbb{E}_t\Big[h_j^4\Big]\le C\varepsilon^{-4(d+1)},\quad
\mathbb{E}_t\Big[h_m^2h_n^2\Big]\le C\varepsilon^{-4(d+1)}.
$$
Therefore, we have
$$
\mathbb{P}_t(\mathcal{N}_{\mu,t}^1)\le C\Big(\frac{N^\mu}{N}\Big)^4\varepsilon^{-4(d+1)}\Big(N+\frac{N(N-1)}{2}\Big)
\le C\varepsilon^{-4(d+1)}N^{-(2-4\mu)},
$$
then
$$
\mathbb{P}_0(\mathcal{N}_{\mu})=\mathbb{P}_t(\mathcal{N}_{\mu,t})
\le N\mathbb{P}_t(\mathcal{N}_{\mu,t}^1)
\le C\varepsilon^{-4(d+1)}N^{-(1-4\mu)}.
$$
\end{proof}
\begin{lemma}\label{lm3.4}
Let $\mathcal{N}_\alpha,\,\mathcal{N}_\kappa,\,\mathcal{N}_\gamma,\,\mathcal{N}_\eta,\,\mathcal{N}_\mu$ be defined as in \eqref{def3.1-1} -- \eqref{3.2-4}. Suppose that $f^{\varepsilon,\delta}(t,x,v)$ satisfies Assumptions \ref{assum}. Then there exists a constant $C>0$ such that
{\small\begin{align*}
&\Big|\Big(V_t^{\varepsilon,\delta},\Psi^{\varepsilon,\delta}(X_t^{\varepsilon,\delta})
+\Gamma(X_t^{\varepsilon,\delta},V_t^{\varepsilon,\delta})+\Phi^{\varepsilon,\delta}(X_t^{\varepsilon,\delta})\Big)
-\Big(\overline{V}_t^{\varepsilon,\delta},\overline{\Psi}^{\varepsilon,\delta}(\overline{X}_t^{\varepsilon,\delta})
+\Gamma(\overline{X}_t^{\varepsilon,\delta},\overline{V}_t^{\varepsilon,\delta})+\overline{\Phi}^{\varepsilon,\delta}(X_t^{\varepsilon,\delta})\Big)\Big|_\infty\\
&\le \Big(C+CR\delta^{-1}+CR\delta^{-1}N^{-\alpha}\varepsilon^{-(d+1)}\Big)S_t(X,V)N^{-\alpha}+N^{-\kappa}+N^{-\eta}
\end{align*}}
for all initial data $(X,V)\in(\mathcal{N}_\alpha\cup\mathcal{N}_\kappa\cup\mathcal{N}_\gamma\cup\mathcal{N}_\eta\cup\mathcal{N}_\mu)^c$.
\end{lemma}
\begin{proof} A series of calculations give that
{\small\begin{align*}
&\Big|\Big(V_t^{\varepsilon,\delta},\Psi^{\varepsilon,\delta}(X_t^{\varepsilon,\delta})
+\Gamma(X_t^{\varepsilon,\delta},V_t^{\varepsilon,\delta})+\Phi^{\varepsilon,\delta}(X_t^{\varepsilon,\delta})\Big)
-\Big(\overline{V}_t^{\varepsilon,\delta},\overline{\Psi}^{\varepsilon,\delta}(\overline{X}_t^{\varepsilon,\delta})
+\Gamma(\overline{X}_t^{\varepsilon,\delta},\overline{V}_t^{\varepsilon,\delta})+\overline{\Phi}^{\varepsilon,\delta}(X_t^{\varepsilon,\delta})\Big)\Big|_\infty\\
&\le\Big|V_t^{\varepsilon,\delta}-\overline{V}_t^{\varepsilon,\delta}\Big|_\infty
+\Big|\Psi^{\varepsilon,\delta}(X_t^{\varepsilon,\delta})-\overline{\Psi}^{\varepsilon,\delta}(\overline{X}_t^{\varepsilon,\delta})\Big|_\infty
+\Big|\Gamma(X_t^{\varepsilon,\delta},V_t^{\varepsilon,\delta})-\Gamma(\overline{X}_t^{\varepsilon,\delta},\overline{V}_t^{\varepsilon,\delta})\Big|_\infty\\
&~~~+\Big|\Phi^{\varepsilon,\delta}(X_t^{\varepsilon,\delta})-\overline{\Phi}^{\varepsilon,\delta}(\overline{X}_t^{\varepsilon,\delta})\Big|_\infty\\
&\le\Big|V_t^{\varepsilon,\delta}-\overline{V}_t^{\varepsilon,\delta}\Big|_\infty
+\Big|\Psi^{\varepsilon,\delta}(X_t^{\varepsilon,\delta})-\Psi^{\varepsilon,\delta}(\overline{X}_t^{\varepsilon,\delta})\Big|_\infty
+\Big|\Psi^{\varepsilon,\delta}(\overline{X}_t^{\varepsilon,\delta})-\overline{\Psi}^{\varepsilon,\delta}(\overline{X}_t^{\varepsilon,\delta})\Big|_\infty\\
&~~~+\Big|\Gamma(X_t^{\varepsilon,\delta},V_t^{\varepsilon,\delta})-\Gamma(\overline{X}_t^{\varepsilon,\delta},\overline{V}_t^{\varepsilon,\delta})\Big|_\infty
+\Big|\Phi^{\varepsilon,\delta}(X_t^{\varepsilon,\delta})-\Phi^{\varepsilon,\delta}(\overline{X}_t^{\varepsilon,\delta})\Big|_\infty
+\Big|\Phi^{\varepsilon,\delta}(\overline{X}_t^{\varepsilon,\delta})-\overline{\Phi}^{\varepsilon,\delta}(\overline{X}_t^{\varepsilon,\delta})\Big|_\infty\\
&=:|I_1|+|I_2|+|I_3|+|I_4|++|I_5|+|I_6|.
\end{align*}}

Next, we estimate from $|I_1|$ to $|I_6|$.

For $|I_1|$. Since $(X,V)\notin\mathcal{N}_\alpha$,
  $$|I_1|:=\Big|V_t^{\varepsilon,\delta}-\overline{V}_t^{\varepsilon,\delta}\Big|_\infty\le S_t(X,V)N^{-\alpha}.$$

For $|I_2|$. \begin{align*} \Big|(\Psi^{\varepsilon}(X_t^{\varepsilon,\delta}))_i-(\Psi^{\varepsilon}(\overline{X}_t^{\varepsilon,\delta}))_i\Big|
      &\le\Big|\frac{\lambda}{N}
      \sum_{j=1}^N
      \nabla_xW^{\varepsilon}(x_i^{\varepsilon,\delta}-x_j^{\varepsilon,\delta})
      -\frac{\lambda}{N}
      \sum_{i\neq j}^N
      \nabla_xW^{\varepsilon}(\overline{x}_i^{\varepsilon,\delta}-\overline{x}_j^{\varepsilon,\delta})\Big|\\
      &\le\frac{\lambda}{N} \sum_{j=1}^N
      |q^{\varepsilon}(\overline{x}_i^{\varepsilon,\delta}-\overline{x}_j^{\varepsilon,\delta})|
      \Big(|x_i^{\varepsilon,\delta}-\overline{x}_i^{\varepsilon,\delta}|
      +|x_j^{\varepsilon,\delta}-\overline{x}_j^{\varepsilon,\delta}|\Big).
      \end{align*}
      Since $\alpha>\theta$ and $(X,V)\notin\mathcal{N}_\alpha$, for any $1\le i,j\le N$, we have that
      $$
      |x_i^{\varepsilon,\delta}-\overline{x}_i^{\varepsilon,\delta}|\le N^{-\alpha}<\varepsilon
      \quad{\rm and~}\quad
      |x_j^{\varepsilon,\delta}-\overline{x}_j^{\varepsilon,\delta}|\le N^{-\alpha}<\varepsilon,
      $$
      which satisfy Assumption \ref{lm2.1} $(ii)$.
      Moreover because $(X,V)\notin\mathcal{N}_\gamma$, we get
      \begin{align*}
       \Big|(\Psi^{\varepsilon,\delta}(X_t^{\varepsilon,\delta}))_i-(\Psi^{\varepsilon,\delta}(\overline{X}_t^{\varepsilon,\delta}))_i\Big|
       &\le 2\lambda\Big|(q^{\varepsilon}(\overline{X}_t^{\varepsilon,\delta}))_i\Big| N^{-\alpha}
       \le2\lambda(\|q^{\varepsilon}\ast\rho^{\varepsilon,\delta}\|_\infty+N^{-\gamma})N^{-\alpha}.
       \end{align*}
      From $(iii)$ of Assumption \ref{assum}, we have
      \begin{align*}
      \|q^{\varepsilon}\ast\rho^{\varepsilon,\delta}\|_\infty&=\|\int_{R^d}q^{\varepsilon}(x-y)\rho^{\varepsilon,\delta}(s,y,v)dy\|_\infty\\
      &\le C\|\int_{\{|x-y|\ge3\varepsilon\}\times\R^d}\frac{1}{|x-y|^d}f^{\varepsilon,\delta}(s,y,v)dy\|_\infty\\
      &~~~+ C\|\int_{\{|x-y|<3\varepsilon\}\times\R^d}\varepsilon^{-d} f^{\varepsilon,\delta}(s,y,v)dy\|_\infty\\
      &\le C.
      \end{align*}
      And thus
      $$
      |I_2|:=\Big|\Psi^{\varepsilon,\delta}(X_t^{\varepsilon,\delta})-\Psi^{\varepsilon,\delta}(\overline{X}_t^{\varepsilon,\delta})\Big|_\infty
      \le CS_t(X,V) N^{-\alpha}.
      $$

 For $|I_3|$. Since $(X,V)\notin\mathcal{N}_\kappa$, it follows directly
      $$
      |I_3|:=\Big|\Psi^{\varepsilon,\delta}(\overline{X}_t^{\varepsilon,\delta})-\overline{\Psi}^{\varepsilon,\delta}(\overline{X}_t^{\varepsilon,\delta})\Big|_\infty
      \le N^{-\kappa}.
      $$

   For $|I_4|$. Since $G$ is Lipschitz continuous, we have for each $1\le i\le N$

      $$
      \Big|G(x_i^{\varepsilon,\delta},v_i^{\varepsilon,\delta})
      -G(\overline{x}_i^{\varepsilon,\delta},\overline{v}_i^{\varepsilon,\delta})\Big|
      \le L
      \Big|(x_i^{\varepsilon,\delta},v_i^{\varepsilon,\delta})-(\overline{x}_i^{\varepsilon,\delta},\overline{v}_i^{\varepsilon,\delta})\Big|.
      $$
      Since $(X,V)\notin\mathcal{N}_\alpha$, we obtain
      $$
      |I_4|:=\Big|\Gamma(X_t^{\varepsilon,\delta},V_t^{\varepsilon,\delta})-\Gamma(\overline{X}_t^{\varepsilon,\delta},\overline{V}_t^{\varepsilon,\delta})\Big|_\infty
      \le LS_t(X,V)N^{-\alpha}.
      $$

  For $|I_5|$. For each $1\le i\le N$ and the definition of $\phi_R$
      {\footnotesize\begin{align*}
      &|u^{\varepsilon,\delta}(x_i^{\varepsilon,\delta})-u^{\varepsilon,\delta}(\overline{x}_i^{\varepsilon,\delta})|\\
      &=\Big|\frac{\frac{1}{N}\sum\limits_{j=1}^N \phi_R(v_j^{\varepsilon,\delta})\phi^{\varepsilon}(x_i^{\varepsilon,\delta}-x_j^{\varepsilon,\delta})}
{\frac{1}{N}\sum\limits_{j=1}^N\phi^{\varepsilon}(x_i^{\varepsilon,\delta}-x_j^{\varepsilon,\delta})+\delta}
-\frac{\frac{1}{N}\sum\limits_{j=1}^N \phi_R(\overline{v}_j^{\varepsilon,\delta})\phi^{\varepsilon}(\overline{x}_i^{\varepsilon,\delta}-\overline{x}_j^{\varepsilon,\delta})}
{\frac{1}{N}\sum\limits_{j=1}^N\phi^{\varepsilon}(\overline{x}_i^{\varepsilon,\delta}-\overline{x}_j^{\varepsilon,\delta})+\delta}\Big|\\
&\le\Big|\frac{\Big(\frac{1}{N}\sum\limits_{j=1}^N\phi^{\varepsilon}(\overline{x}_i^{\varepsilon,\delta}-\overline{x}_j^{\varepsilon,\delta})+\delta\Big)
\Big(\frac{1}{N}\sum\limits_{j=1}^N(\phi_R(v_j^{\varepsilon,\delta})-\phi_R(\overline{v}_j^{\varepsilon,\delta}))\phi^{\varepsilon}(x_i^{\varepsilon,\delta}-x_j^{\varepsilon,\delta})\Big)}
{\Big(\frac{1}{N}\sum\limits_{j=1}^N\phi^{\varepsilon}(x_i^{\varepsilon,\delta}-x_j^{\varepsilon,\delta})+\delta\Big)\Big(\frac{1}{N}\sum\limits_{j=1}^N\phi^{\varepsilon}(\overline{x}_i^{\varepsilon,\delta}-\overline{x}_j^{\varepsilon,\delta})+\delta\Big)}\Big|\\
&~~~+\Big|\frac{\Big(\frac{1}{N}\sum\limits_{j=1}^N\phi^{\varepsilon}(\overline{x}_i^{\varepsilon,\delta}-\overline{x}_j^{\varepsilon,\delta})+\delta\Big)
\Big(\frac{1}{N}\sum\limits_{j=1}^N\phi_R(\overline{v}_j^{\varepsilon,\delta})(\phi^{\varepsilon}(x_i^{\varepsilon,\delta}-x_j^{\varepsilon,\delta})-\phi^{\varepsilon}(\overline{x}_i^{\varepsilon,\delta}-\overline{x}_j^{\varepsilon,\delta}))\Big)}
{\Big(\frac{1}{N}\sum\limits_{j=1}^N\phi^{\varepsilon}(x_i^{\varepsilon,\delta}-x_j^{\varepsilon,\delta})+\delta\Big)\Big(\frac{1}{N}\sum\limits_{j=1}^N\phi^{\varepsilon}(\overline{x}_i^{\varepsilon,\delta}-\overline{x}_j^{\varepsilon,\delta})+\delta\Big)}\Big|\\
&~~~+\Big|\frac{\frac{1}{N}\sum\limits_{j=1}^N \phi_R(\overline{v}_j^{\varepsilon,\delta})\phi^{\varepsilon}(\overline{x}_i^{\varepsilon,\delta}-\overline{x}_j^{\varepsilon,\delta})\Big(\frac{1}{N}\sum\limits_{j=1}^N(\phi^{\varepsilon}(\overline{x}_i^{\varepsilon,\delta}-\overline{x}_j^{\varepsilon,\delta})-\phi^{\varepsilon}(x_i^{\varepsilon,\delta}-x_j^{\varepsilon,\delta}))\Big)}{\Big(\frac{1}{N}\sum\limits_{j=1}^N\phi^{\varepsilon}(x_i^{\varepsilon,\delta}-x_j^{\varepsilon,\delta})+\delta\Big)\Big(\frac{1}{N}\sum\limits_{j=1}^N\phi^{\varepsilon}(\overline{x}_i^{\varepsilon,\delta}-\overline{x}_j^{\varepsilon,\delta})+\delta\Big)}\Big|\\
&<\max\limits_{1\le j\le N}|v_j^{\varepsilon,\delta}-\overline{v}_j^{\varepsilon,\delta}|
+\max\limits_{1\le j\le N}|\phi_R(\overline{v}_j^{\varepsilon,\delta})|\delta^{-1}\frac{1}{N}\sum\limits_{j=1}^N|\phi^{\varepsilon}(x_i^{\varepsilon,\delta}-x_j^{\varepsilon,\delta})-\phi^{\varepsilon}(\overline{x}_i^{\varepsilon,\delta}-\overline{x}_j^{\varepsilon,\delta})|.
      \end{align*}}
Using Taylor formula, we get
{\footnotesize\begin{align}\label{a}
\frac{1}{N}\sum\limits_{j=1}^N|\phi^{\varepsilon}(x_i^{\varepsilon,\delta}-x_j^{\varepsilon,\delta})-\phi^{\varepsilon}(\overline{x}_i^{\varepsilon,\delta}-\overline{x}_j^{\varepsilon,\delta})|
&\le \frac{C}{N}\sum\limits_{j=1}^N|(\phi^{\varepsilon})'(\overline{x}_i^{\varepsilon,\delta}-\overline{x}_j^{\varepsilon,\delta})||(\overline{x}_i^{\varepsilon,\delta}-\overline{x}_j^{\varepsilon,\delta})-(x_i^{\varepsilon,\delta}-x_j^{\varepsilon,\delta})|\nonumber\\
&~~~+\frac{C}{N}\sum\limits_{j=1}^N\Big|\frac{(\phi^{\varepsilon})''(\overline{x}_i^{\varepsilon,\delta}-\overline{x}_j^{\varepsilon,\delta})}{2}\Big||(\overline{x}_i^{\varepsilon,\delta}-\overline{x}_j^{\varepsilon,\delta})-(x_i^{\varepsilon,\delta}-x_j^{\varepsilon,\delta})|^2\nonumber\\
&~~~+o(|(\overline{x}_i^{\varepsilon,\delta}-\overline{x}_j^{\varepsilon,\delta})-(x_i^{\varepsilon,\delta}-x_j^{\varepsilon,\delta})|^2).
\end{align}}
Since $(X,V)\notin\mathcal{N}_\alpha$, \eqref{a} is written that
      {\footnotesize$$
\frac{1}{N}\sum\limits_{j=1}^N|\phi^{\varepsilon}(x_i^{\varepsilon,\delta}-x_j^{\varepsilon,\delta})-\phi^{\varepsilon}(\overline{x}_i^{\varepsilon,\delta}-\overline{x}_j^{\varepsilon,\delta})|
\le \frac{C}{N}\sum\limits_{j=1}^N|(\phi^{\varepsilon})'(\overline{x}_i^{\varepsilon,\delta}-\overline{x}_j^{\varepsilon,\delta})|N^{-\alpha}
+\frac{C}{N}\sum\limits_{j=1}^N\Big|\frac{(\phi^{\varepsilon})''(\overline{x}_i^{\varepsilon,\delta}-\overline{x}_j^{\varepsilon,\delta})}{2}\Big|N^{-2\alpha}.
      $$}
    Because $(X,V)\notin\mathcal{N}_\mu$ and $(ii)$ of Assumption  \ref{assum}, we get
    \begin{align*}
    \frac{1}{N}\sum\limits_{j=1}^N|(\phi^{\varepsilon'})(\overline{x}_i^{\varepsilon,\delta}-\overline{x}_j^{\varepsilon,\delta})|
    \le\|\int_{\mathbb{R}^d}|(\phi^{\varepsilon'})(\overline{x}_i^{\varepsilon,\delta}-y)|\rho^{\varepsilon,\delta}(t,y)dy\|_\infty+N^{-\mu}\le C,
    \end{align*}
    and thus
    \begin{align*}
    |I_5|:=\Big|\Phi^{\varepsilon,\delta}(X_t^{\varepsilon,\delta})-\Phi^{\varepsilon,\delta}(\overline{X}_t^{\varepsilon,\delta})\Big|_\infty
    \le (1+CR\delta^{-1}+CR\delta^{-1}N^{-\alpha}\varepsilon^{-(d+1)}) S_t(X,V)N^{-\alpha}.
 \end{align*}

  For $|I_6|$. Since $(X,V)\notin\mathcal{N}_\eta$, it follows directly
      $$
      |I_6|:=\Big|\Phi^{\varepsilon,\delta}(\overline{X}_t^{\varepsilon,\delta})-\overline{\Phi}^{\varepsilon,\delta}(\overline{X}_t^{\varepsilon,\delta})\Big|_\infty
      \le N^{-\eta}.
      $$

Combining all the six terms, we have
{\footnotesize\begin{align*}
&\Big|\Big(V_t^{\varepsilon,\delta},\Psi^{\varepsilon,\delta}(X_t^{\varepsilon,\delta})
+\Gamma(X_t^{\varepsilon,\delta},V_t^{\varepsilon,\delta})+\Phi^{\varepsilon,\delta}(X_t^{\varepsilon,\delta})\Big)
-\Big(\overline{V}_t^{\varepsilon,\delta},\overline{\Psi}^{\varepsilon,\delta}(\overline{X}_t^{\varepsilon,\delta})
+\Gamma(\overline{X}_t^{\varepsilon,\delta},\overline{V}_t^{\varepsilon,\delta})+\overline{\Phi}^{\varepsilon,\delta}(X_t^{\varepsilon,\delta})\Big)\Big|_\infty\\
&\le \Big(C+C R\delta^{-1}+C R\delta^{-1}N^{-\alpha}\varepsilon^{-(d+1)}\Big)S_t(X,V)N^{-\alpha}+N^{-\kappa}+N^{-\eta}
\end{align*}}
for all $(X,V)\in(\mathcal{N}_\alpha\cup\mathcal{N}_\kappa\cup\mathcal{N}_\gamma\cup\mathcal{N}_\eta\cup\mathcal{N}_\mu)^c$.
\end{proof}
Now using Lemmas \ref{lm3.1} -- \ref{lm3.4} to complete the prove  of Theorem \ref{th3.1}:
\begin{proof}[ Proof of Theorem \ref{th3.1}]
From \eqref{def2.1-1} and \eqref{def2.2-1}, we know
$$
(X_{t+dt}^{\varepsilon,\delta},V_{t+dt}^{\varepsilon,\delta})
=(X_t^{\varepsilon,\delta},V_t^{\varepsilon,\delta})
+\Big(V_t^{\varepsilon,\delta},\Psi^{\varepsilon,\delta}(X_t^{\varepsilon,\delta})
+\Gamma(X_t^{\varepsilon,\delta},V_t^{\varepsilon,\delta})
+\Phi^{\varepsilon,\delta}(X_t^{\varepsilon,\delta})\Big)dt+o(dt),
$$
and
$$
(\overline{X}_{t+dt}^{\varepsilon,\delta},\overline{V}_{t+dt}^{\varepsilon,\delta})
=(\overline{X}_t^{\varepsilon,\delta},\overline{V}_t^{\varepsilon,\delta})
+\Big(\overline{V}_t^{\varepsilon,\delta},\overline{\Psi}^{\varepsilon,\delta}(\overline{X}_t^{\varepsilon,\delta})
+\Gamma(\overline{X}_t^{\varepsilon,\delta},\overline{V}_t^{\varepsilon,\delta})
+\overline{\Phi}^{\varepsilon,\delta}(\overline{X}_t^{\varepsilon,\delta})\Big)dt+o(dt).
$$
Then
{\footnotesize\begin{align*}
&\Big|(X_{t+dt}^{\varepsilon,\delta},V_{t+dt}^{\varepsilon,\delta})
-(\overline{X}_{t+dt}^{\varepsilon,\delta},\overline{V}_{t+dt}^{\varepsilon,\delta})\Big|_\infty\\
&\le\Big|(X_t^{\varepsilon,\delta},V_t^{\varepsilon,\delta})-(\overline{X}_t^{\varepsilon,\delta},\overline{V}_t^{\varepsilon,\delta})\Big|_\infty+\Big|\Big(V_t^{\varepsilon,\delta},\Psi^{\varepsilon,\delta}(X_t^{\varepsilon,\delta})+\Gamma(X_t^{\varepsilon,\delta},V_t^{\varepsilon,\delta})
+\Phi^{\varepsilon}(X_t^{\varepsilon,\delta})\Big)\\
&~~~-\Big(\overline{V}_t^{\varepsilon,\delta},\overline{\Psi}^{\varepsilon,\delta}(\overline{X}_t^{\varepsilon,\delta}))
+\Gamma(\overline{X}_t^{\varepsilon,\delta},\overline{V}_t^{\varepsilon,\delta})
+\overline{\Phi}^{\varepsilon,\delta}(\overline{X}_t^{\varepsilon,\delta})\Big)\Big|_\infty dt+o(dt),
\end{align*}}
i.e.,
{\footnotesize\begin{align*}
S_{t+dt}-S_t
&\le\Big|\Big(V_t^{\varepsilon,\delta},\Psi^{\varepsilon,\delta}(X_t^{\varepsilon,\delta})
+\Gamma(X_t^{\varepsilon,\delta},V_t^{\varepsilon,\delta})
+\Phi^{\varepsilon}(X_t^{\varepsilon,\delta},V_t^{\varepsilon,\delta})\Big)\\
&~~~-\Big(\overline{V}_t^{\varepsilon,\delta},\overline{\Psi}^{\varepsilon,\delta}(\overline{X}_t^{\varepsilon,\delta}))
+\Gamma(\overline{X}_t^{\varepsilon,\delta},\overline{V}_t^{\varepsilon,\delta})
+\overline{\Phi}^{\varepsilon,\delta}(\overline{X}_t^{\varepsilon,\delta})\Big)\Big|_\infty N^{\alpha}dt+o(dt).
\end{align*}}
Taking the expecting over both sides yields
{\small\begin{align*}
\mathbb{E}_0\Big[S_{t+dt}-S_t\Big]
&=\mathbb{E}_0\Big[S_{t+dt}-S_t|\mathcal{N}_\alpha\Big]+\Big[S_{t+dt}-S_t|\mathcal{N}_\alpha^c\Big]\\
&\le\mathbb{E}_0\Big[S_{t+dt}-S_t|(\mathcal{N}_\kappa\cup\mathcal{N}_\gamma\cup\mathcal{N}_\eta\cup\mathcal{N}_\mu)\setminus \mathcal{N}_\alpha\Big]
+\mathbb{E}_0\Big[S_{t+dt}-S_t|(\mathcal{N}_\alpha\cup\mathcal{N}_\kappa\cup\mathcal{N}_\gamma\cup\mathcal{N}_\eta\cup\mathcal{N}_\mu)^c\Big]\\
&\le\mathbb{E}_0\Big[\Big|V_t^{\varepsilon,\delta}-\overline{V}_t^{\varepsilon,\delta}\Big|_\infty\Big|(\mathcal{N}_\kappa\cup\mathcal{N}_\gamma\cup\mathcal{N}_\eta\cup\mathcal{N}_\mu)\setminus \mathcal{N}_\alpha\Big]N^{\alpha}dt\\
&~~~+\mathbb{E}_0\Big[\Big|\Psi^{\varepsilon,\delta}(X_t^{\varepsilon,\delta})-\overline{\Psi}^{\varepsilon,\delta}(\overline{X}_t^{\varepsilon,\delta})\Big|_\infty\Big|(\mathcal{N}_\kappa\cup\mathcal{N}_\gamma\cup\mathcal{N}_\eta\cup\mathcal{N}_\mu)\setminus \mathcal{N}_\alpha\Big]N^{\alpha}dt\\
&~~~+\mathbb{E}_0\Big[\Big|\Gamma(X_t^{\varepsilon,\delta},V_t^{\varepsilon,\delta})-\Gamma(\overline{X}_t^{\varepsilon,\delta},\overline{V}_t^{\varepsilon,\delta})\Big|_\infty\Big|(\mathcal{N}_\kappa\cup\mathcal{N}_\gamma\cup\mathcal{N}_\eta\cup\mathcal{N}_\mu)\setminus \mathcal{N}_\alpha\Big]N^{\alpha}dt\\
&~~~+\mathbb{E}_0\Big[\Big|\Phi^{\varepsilon}(X_t^{\varepsilon,\delta})-\overline{\Phi}^{\varepsilon,\delta}(\overline{X}_t^{\varepsilon,\delta})\Big|_\infty\Big|(\mathcal{N}_\kappa\cup\mathcal{N}_\gamma\cup\mathcal{N}_\eta\cup\mathcal{N}_\mu)\setminus \mathcal{N}_\alpha\Big]N^{\alpha}dt\\
&~~~+\mathbb{E}_0\Big[S_{t+dt}-S_t|(\mathcal{N}_\alpha\cup\mathcal{N}_\kappa\cup\mathcal{N}_\gamma\cup\mathcal{N}_\eta)^c\Big]
+o(dt)\\
&=:J_1+J_2+J_3+J_4+J_5+o(dt).
\end{align*}}

 For $J_1$.
 Since $(X,V)\notin\mathcal{N}_\alpha$, we get
 {\footnotesize\begin{align*}
  J_1=\mathbb{E}_0\Big[\Big|V_t^{\varepsilon,\delta}-\overline{V}_t^{\varepsilon,\delta}\Big|_\infty
  \Big|(\mathcal{N}_\kappa\cup\mathcal{N}_\gamma\cup\mathcal{N}_\eta\cup\mathcal{N}_\mu)\setminus \mathcal{N}_\alpha\Big]N^\alpha dt
  \le N^{-\alpha}\Big(\mathbb{P}_0(\mathcal{N}_\kappa)+\mathbb{P}_0(\mathcal{N}_\gamma)+\mathbb{P}_0(\mathcal{N}_\eta)+\mathbb{P}_0(\mathcal{N}_\mu)\Big)
  N^\alpha dt.
\end{align*}}

   For $ J_2$. Due to the definition of $\Psi^{\varepsilon,\delta},\overline{\Psi}^{\varepsilon,\delta}$ as well as the boundedness of $\nabla_xW^{\varepsilon}$, we achieve
      \begin{align*}
      J_2&\le(\|\nabla_xW^{\varepsilon}\|_\infty
      +\|\nabla_xW^{\varepsilon}\ast\rho^{\varepsilon,\delta}\|_\infty)
      \Big(\mathbb{P}_0(\mathcal{N}_\kappa)+\mathbb{P}_0(\mathcal{N}_\gamma)+\mathbb{P}_0(\mathcal{N}_\eta)+\mathbb{P}_0(\mathcal{N}_\mu)\Big)
  N^\alpha dt\\
  &\le C\varepsilon^{-(d-1)}\Big(\mathbb{P}_0(\mathcal{N}_\kappa)+\mathbb{P}_0(\mathcal{N}_\gamma)+\mathbb{P}_0(\mathcal{N}_\eta)+\mathbb{P}_0(\mathcal{N}_\mu)\Big)
  N^\alpha dt.
      \end{align*}

   For $J_3$. Since $G(x,v)$ is Lipschitz continuous and $(X,V)\notin\mathcal{N}_\alpha$, we have
      \begin{align*}
      J_3&:=\mathbb{E}_0\Big[\Big|\Gamma(X_t^{\varepsilon,\delta},V_t^{\varepsilon,\delta})-\Gamma(\overline{X}_t^{\varepsilon,\delta},\overline{V}_t^{\varepsilon,\delta})\Big|_\infty\Big|(\mathcal{N}_\kappa\cup\mathcal{N}_\gamma\cup\mathcal{N}_\eta\cup\mathcal{N}_\mu)\setminus \mathcal{N}_\alpha\Big]N^{\alpha}dt\\
      &\le L \mathbb{E}_0\Big[\Big|(X_t^{\varepsilon,\delta},V_t^{\varepsilon,\delta})-(\overline{X}_t^{\varepsilon,\delta},\overline{V}_t^{\varepsilon,\delta})\Big|_\infty\Big|(\mathcal{N}_\kappa\cup\mathcal{N}_\gamma\cup\mathcal{N}_\eta\cup\mathcal{N}_\mu)\setminus \mathcal{N}_\alpha\Big]N^{\alpha}dt\\
      &\le LN^{-\alpha}
      \Big(\mathbb{P}_0(\mathcal{N}_\kappa)+\mathbb{P}_0(\mathcal{N}_\gamma)+\mathbb{P}_0(\mathcal{N}_\eta)+\mathbb{P}_0(\mathcal{N}_\mu)\Big)
  N^\alpha dt.
      \end{align*}

  For $J_4$. Due to the definition of $\Phi^{\varepsilon,\delta}$, we achieve
  \begin{align*}
  \|u^{\varepsilon,\delta}\|_\infty&=\Big|\frac{\frac{1}{N}\sum\limits_{j=1}^N \phi_R(v_j^{\varepsilon,\delta})\phi^{\varepsilon}(x_i^{\varepsilon,\delta}-x_j^{\varepsilon,\delta})}
{\frac{1}{N}\sum\limits_{j=1}^N\phi^{\varepsilon}(x_i^{\varepsilon,\delta}-x_j^{\varepsilon,\delta})+\delta}\Big|\\
&\le\max\limits_{1\le j\le N}|\phi_R(v_j^{\varepsilon,\delta})|\Big|\frac{\frac{1}{N}\sum\limits_{j=1}^N \phi^{\varepsilon}(x_i^{\varepsilon,\delta}-x_j^{\varepsilon,\delta})}
{\frac{1}{N}\sum\limits_{j=1}^N\phi^{\varepsilon}(x_i^{\varepsilon,\delta}-x_j^{\varepsilon,\delta})+\delta}\Big|\\
&<\max\limits_{1\le j\le N}|\phi_R(v_j^{\varepsilon,\delta})|\le2R.
\end{align*}
Similarly, from the definition of $\overline{\Phi}^{\varepsilon,\delta}$ and $(i)$ of Assumption \ref{assum}, we have
\begin{align*}
\|\overline{u}^{\varepsilon,\delta}\|_\infty&=\Big|\frac{\int_{\R^{2d}}v\phi^{\varepsilon}(\overline{x}_i^{\varepsilon,\delta}-y) f^{\varepsilon,\delta}(t,y,v)dvdy}{\int_{\R^{2d}}\phi^{\varepsilon}(\overline{x}_i^{\varepsilon,\delta}-y) f^{\varepsilon,\delta}(t,y,v)dvdy+\delta}\Big|\\
&\le\Big|\frac{\int_{\Rd}\phi^{\varepsilon}(\overline{x}_i^{\varepsilon,\delta}-y)(\int_{\Rd} vf^{\varepsilon,\delta}(t,y,v)dy)}{\delta}\Big|\\
&\le\frac{C}{\delta}.
\end{align*}
Then
  \begin{align*}
  J_4&\le\beta(\|u^{\varepsilon,\delta}\|_\infty+\|\overline{u}^{\varepsilon,\delta}\|_\infty)\Big(\mathbb{P}_0(\mathcal{N}_\kappa)+\mathbb{P}_0(\mathcal{N}_\gamma)+\mathbb{P}_0(\mathcal{N}_\eta)+\mathbb{P}_0(\mathcal{N}_\mu)\Big)
  N^\alpha dt\\
  &\le C\beta(R+\frac{1}{\delta})\Big(\mathbb{P}_0(\mathcal{N}_\kappa)+\mathbb{P}_0(\mathcal{N}_\gamma)+\mathbb{P}_0(\mathcal{N}_\eta)+\mathbb{P}_0(\mathcal{N}_\mu)\Big)
  N^\alpha dt.
  \end{align*}

 For $J_5$. From Lemma \ref{lm3.4}, we obtain
  \begin{align*}
  J_5&:=\mathbb{E}_0\Big[S_{t+dt}-S_t|(\mathcal{N}_\alpha\cup\mathcal{N}_\kappa\cup\mathcal{N}_\gamma\cup\mathcal{N}_\eta\cup\mathcal{N}_\mu)^c\Big]\\
  &\le \Big((C+CR\delta^{-1}+CR\delta^{-1}N^{-\alpha}\varepsilon^{-(d+1)})\mathbb{E}_0[S_t]N^{-\alpha}+N^{-\kappa}+N^{-\eta}\Big)
  N^{\alpha}dt+o(dt)\\
  &=(C+CR\delta^{-1}+CR\delta^{-1}N^{-\alpha}\varepsilon^{-(d+1)})\mathbb{E}_0[S_t]dt+N^{\alpha-\kappa}dt+N^{\alpha-\eta}dt+o(dt).
  \end{align*}
  Therefore, we can determine the estimate
  {\small\begin{align*}
  \mathbb{E}_0\Big[S_{t+dt}\Big]-\mathbb{E}_0[S_t]
  &\le\mathbb{E}_0\Big[S_{t+dt}-S_t\Big]\\
  &\le C\Big(N^{-\alpha}+\varepsilon^{-(d-1)}+R+\frac{1}{\delta}\Big)
   \Big(\mathbb{P}_0(\mathcal{N}_\kappa)+\mathbb{P}_0(\mathcal{N}_\gamma)+\mathbb{P}_0(\mathcal{N}_\eta)+\mathbb{P}_0(\mathcal{N}_\mu)\Big)
  N^\alpha dt\\
  &~~~+(C+CR\delta^{-1}+CR\delta^{-1}N^{-\alpha}\varepsilon^{-(d+1)})\mathbb{E}_0[S_t]dt+N^{\alpha-\kappa}dt+N^{\alpha-\eta}dt+o(dt)\\
  &\le (C+CR\delta^{-1}+CR\delta^{-1}N^{-\alpha}\varepsilon^{-(d+1)})\mathbb{E}_0[S_t]dt +C\Big(N^{-\alpha}+\varepsilon^{-(d-1)}+R+\frac{1}{\delta}\Big)\\
  &~~~~~~~\cdot\max\Big\{\mathbb{P}_0(\mathcal{N}_\kappa)N^\alpha,\,\mathbb{P}_0(\mathcal{N}_\gamma)N^\alpha,\,\mathbb{P}_0(\mathcal{N}_\eta)N^\alpha,\,\mathbb{P}_0(\mathcal{N}_\mu)N^\alpha,\,N^{\alpha-\kappa},\,N^{\alpha-\eta}\Big\}dt+o(dt).
\end{align*}}
Equivalently, we have
\begin{align*}
\frac{d}{dt}\mathbb{E}_0[S_t]
&\le (C+CR\delta^{-1}+CR\delta^{-1}N^{-\alpha}\varepsilon^{-(d+1)})\mathbb{E}_0[S_t]
+C\Big (N^{-\alpha}+\varepsilon^{-(d-1)}+R+\frac{1}{\delta}\Big)\\
&~~~~~\cdot\max\Big\{\mathbb{P}_0(\mathcal{N}_\kappa)N^\alpha,\,\mathbb{P}_0(\mathcal{N}_\gamma)N^\alpha,\,\mathbb{P}_0(\mathcal{N}_\eta)N^\alpha,\,\mathbb{P}_0(\mathcal{N}_\mu)N^\alpha,\,N^{\alpha-\kappa},\,N^{\alpha-\eta}\Big\}.
\end{align*}
Hence, Gronwall's inequality yields
{\footnotesize\begin{align*}
\mathbb{E}_0[S_t]&\le e^{(C+CR\delta^{-1}+CR\delta^{-1}N^{-\alpha}\varepsilon^{-(d+1)})t}
\\&~~~\cdot \Big[C\Big(N^{-\alpha}+\varepsilon^{-(d-1)}+R+\frac{1}{\delta}\Big)
\cdot\max\Big\{\mathbb{P}_0(\mathcal{N}_\kappa)N^\alpha,\,\mathbb{P}_0(\mathcal{N}_\gamma)N^\alpha,\,\mathbb{P}_0(\mathcal{N}_\eta)N^\alpha,\,\mathbb{P}_0(\mathcal{N}_\mu)N^\alpha,\,N^{\alpha-\kappa},\,N^{\alpha-\eta}\Big\}\Big].
\end{align*}}
We choose $\varepsilon=N^{-\theta}, R=\delta^{-1}=\sqrt{\vartheta \ln(N)}$, then we achieve
$$
\mathbb{E}_0[S_t]\le C \exp\Big\{\Big(C+C\vartheta\ln(N)\Big)t\Big\}\cdot N^{-n}.
$$
\end{proof}



\bibliographystyle{abbrv}
\bibliography{reference}

\end{document}